\documentclass[11pt,leqno,amscd,amssymb]{amsart}
\usepackage[english]{babel}
\usepackage{amsfonts,amssymb,graphicx,amsmath,amsthm,hyperref,fullpage,tikz,amscd,mathrsfs,url}\usepackage{enumitem}
\usetikzlibrary{matrix}
\usetikzlibrary{cd}
\usepackage{wasysym}

\usepackage{mathtools}
\usepackage{mathscinet}
\allowdisplaybreaks

\newcommand{\id}{\operatorname{Id}}
\newcommand{\Hom}{\text{\rm Hom}}

\newcommand{\End}{\operatorname{End}}

\newcommand{\inv}{^{-1}}

\newcommand{\Uq}{U_q(\mathfrak{gl}_n)}
\newcommand{\Hb}{\mathcal{H}_{Q,q}^B(d)}

\renewcommand{\mod}{\operatorname{mod}}
\renewcommand{\id}{\operatorname{id}}
\newcommand{\res}{\operatorname{Res}}

\newcommand{\td}{^{\otimes d}}

\newtheorem{thm}{Theorem}[section] 

\newtheorem{lem}[thm]{Lemma}
\newtheorem{cor}[thm]{Corollary}

\newtheorem{prop}[thm]{Proposition}
\theoremstyle{definition}
\newtheorem{ex}[thm]{Example}

\newtheorem{defn}[thm]{Definition}

\newtheorem{rem}[thm]{Remark}

\newtheorem{assumption}[thm]{Assumption}

\newcommand{\GL}{\operatorname{GL}}
\newcommand{\PBd}{\mathcal{P}_{Q,q}^d}

\newcommand{\half}{\frac{1}{2}}

\newcommand{\SB}{S^B_{Q,q}}
\newcommand{\HB}{\mathcal{H}_{Q,q}^B(d)}
\newcommand{\apde}{\mathcal{AP}^{d,e}_q}
\newcommand{\pmpower}{^a_+\otimes ^b_-}

\begin{document}

\title{Polynomial functors and two-parameter quantum symmetric pairs}
\author{Valentin Buciumas}
\address{School of Mathematics and Physics,
The University of Queensland, 
St. Lucia, QLD 4072, 
Australia}
\email{valentin.buciumas@gmail.com}
\author{Hankyung Ko}
\address{ Department of Mathematics, Uppsala University, Box. 480, SE-75106, Uppsala, Sweden}
\email{hankyung.ko@math.uu.se}
\date{}

\maketitle

\begin{abstract}
We develop a theory of two-parameter quantum polynomial functors. Similar to how (strict) polynomial functors give a new interpretation of polynomial representations of the general linear groups $\GL_n$, the two-parameter polynomial functors give a new interpretation of (polynomial) representations of the quantum symmetric pair $(U_{Q,q}^B(\mathfrak{gl}_n), U_q(\mathfrak{gl}_n) )$ which specializes to type AIII/AIV quantum symmetric pairs. The coideal subalgebra $U_{Q,q}^B(\mathfrak{gl}_n)$ appears in a Schur-Weyl duality with the type B Hecke algebra $\mathcal H^B_{Q,q}(d)$. We endow two-parameter polynomial functors with a cylinder braided structure which we use to construct the two-parameter Schur functors. Our polynomial functors can be precomposed with the quantum polynomial functors of type A producing new examples of action pairs. 
\end{abstract}

\tableofcontents

\section{Introduction}

Polynomial functors are endofunctors on the category of vector spaces that are polynomial on the space of morphisms. They are related to the polynomial representations of $\operatorname{GL}_n$ in the sense that the degree $d$ polynomial functors are equivalent to the degree $d$ representation of $\GL_n$ when $n\geq d$ (this correspondence passes through the Schur algebra). Two quantizations of polynomial functors were developed by Hong and Yacobi~\cite{HY} (first) and by the authors~\cite{BuciumasKo1}. The first category is related to the polynomial representation theory of the quantum group $\Uq$. The second category is related to a ``higher degree'' quantization of $\GL_n$~\cite[Corollary 6.16]{BuciumasKo1}; it is more complicated than the category from~\cite{HY} and was constructed in order to define composition of quantum polynomial functors. Composition is a natural operation on functors which is useful in performing cohomological computations. For example, it enables 
Friedlander and Suslin~\cite{FS} to prove the cohomological finite generation of finite group schemes.

In the present paper we define and study $\emph{two-parameter quantum polynomial functors}$. These polynomial functors are related to the representation theory of a certain coideal subalgebra $U_{Q,q}^B$ (to be defined in Section~\ref{coidealschurdef}) in the same way that classical polynomial functors are related to the representation theory of $\GL_n$. Many of the properties of classical or quantum polynomial functors have (sometimes surprising) analogues for two-parameter polynomial functors, as we show in this paper. 

A quantum symmetric pair is a pair of algebras $B\subset U_q(\mathfrak{g})$ where $\mathfrak{g}$ is a simple Lie algebra and $B$ is constructed from an involution $\theta$ of $\mathfrak{g}$. The subalgebra $B$ has the following property: by restricting the comultiplication $\Delta$ of $U_q(\mathfrak{g})$ to $B$, one obtains a map $\Delta: B \to B \otimes U_q(\mathfrak{g})$. The subalgebra $B$ is also called a coideal subalgebra for this reason.
Such coideal subalgebras have been studied in special cases using solutions of the reflection equation by Noumi, Sugitani, and Dijkhuizen~\cite{Noumi, NS, NSD} and in general by Letzter~\cite{Letzter99, Letzter03}. For more details about quantum symmetric pairs and their applications see the introduction to the paper of Kolb \cite{Kolbaffine} where an affine version of the theory of quantum symmetric pairs is developed.

In this work, we restrict our attention to a specific type of coideal subalgebra $U_{Q,q}^B$
. The motivation for studying this coideal subalgebra is manifold. It is part of a quantum symmetric pair that comes with solutions of the reflection equation and is in (Schur-Weyl) duality with the unequal parameter Hecke algebra of type B. It also plays a major role in many recent works in representation theory. 

We first mention two important independent works where the coideal $U_{Q,q}^B$ and its specializations play a key role. In Bao and Wang~\cite{BW13}, a theory of canonical bases for the coideal subalgebra $U_{q,q}^B$ (denoted by $\mathbf{U}^\iota$ and $\mathbf{U}^\jmath$ in Sections 2.1 and 6.1) is initiated and used to obtain decomposition numbers for the BGG category $\mathcal{O}$ of the Lie superalgebra $\mathfrak{osp}(2m+1|2n)$. 
The coideal at $q=1$ appears as an algebra generated by certain translation functors.

In~\cite{ES}, Ehrig and Stroppel study a 2-categorical action of the coideal $U_{1,q}^B$ on a  parabolic BGG category $\mathcal{O}$ of type D which categorifies an exterior power of the natural representation of the coideal. This process produces canonical bases for the aforementioned coideal modules. 
A Howe duality for the coideal subalgebra surprisingly emerges.

These works started a new wave of interest in quantum symmetric pairs and their applications to representation theory. 
Bao and Wang started a program of studying canonical bases for quantum symmetric pairs \cite{BW13, BWAJM, BKinvolution, Bao, BWInv, BW19} which generalizes Lusztig's theory of canonical basis for $\Uq$~\cite{LusztigCB1}. In related work of Balagovic and Kolb~\cite{BK}, the universal $K$-matrix is constructed for a large class of quantum symmetric pairs including the ones appearing in this work (the universal $K$-matrix for $U_{q,q}^B$ was first written down in \cite[\S 2.5]{BW13}). The universal $K$-matrix produces solutions to the reflection equation similar to how the universal $R$-matrix produces solutions to the Yang-Baxter equation. The search for such solutions of the reflection equation is motivated by the theory of solvable lattice models with U-turn boundary conditions and the study of invariants for braids in a cylinder (according to the work of tom Dieck and H\"{a}ring-Oldenburg~\cite{tD, tDHO, HO}).    

A natural continuation of the work \cite{BW13} is the work of Bao \cite{Bao}, where canonical bases for the specialization $U_{1,q}^B$ are studied, and decomposition numbers for the BGG category $\mathcal{O}$ of $\mathfrak{osp}(2m|2n)$ are obtained. The two papers~\cite{BW13, Bao} establish a Schur-Weyl duality between the coideal subalgebras $U_{q,q}^B$ and $U_{1,q}^B$, and the Hecke algebra $\mathcal{H}_{q,q}^B(d)$ and $\mathcal{H}_{1,q}^B(d)$, respectively (see also \cite{ES} for the $Q=1$ Schur-Weyl duality and \cite{Greenhyper} for a general Schur-Weyl duality without the quantum symmetric pair). The two Schur-Weyl dualities are generalized to a duality between $U_{Q,q}^B$ and $\Hb$ in~\cite{BWW}. The Schur-Weyl duality tells us that a large part of the representation theory of $U_{Q,q}^B$ is encoded in the centralizers of $\HB$ acting on $V_n\td$. This is the starting point of our definition of two-parameter quantum polynomial functors. 

Let $k$ be a field and $Q, q \in k^\times$ and let $\mathcal{C}^B_d$ be the full subcategory of $\HB$-modules (over $k$) of the form $V_n\td$ where the Hecke algebra $\Hb$ acts on a space $V_n\td$ as in equation~\eqref{HeckeactionBC}.
We define two-parameter quantum polynomial functors of degree $d$ as linear functors from the category $\mathcal{C}^B_d$ to the category of vector spaces, that is, we let \[\mathcal{P}_{Q,q}^d=\operatorname{mod}_{\mathcal{C}^B_d}.\]  We prove 
the category $\mathcal{P}^d_{Q,q}$ is equivalent to the category of finite dimensional representations of the two-parameter Schur algebra \[S_{Q,q}^B(n;d):=\operatorname{End}_{\HB}(V_n\td) \] when $n\geq 2d$ is odd. If $Q,q$ are generic, we do not need to require $n$ to be odd (see Setup at the end of the Introduction for what generic means). 
The algebra $S_{Q,q}^B(n;d)$ generalizes the $q$-Schur algebra of Dipper and James and is the main subject of study of the papers \cite{BKLW, LL, LNX}. In particular, \cite[Theorem 3.1.1]{LNX} shows that $\SB(n;d)$ is isomorphic to a direct sum of tensor products of type A $q$-Schur algebras under a small (necessary) restriction on $Q,q$. 

Our construction of polynomial functors and the proof of representability from Section~\ref{sec:quantumpolfunctors} is based on a Schur-Weyl duality and does not use any other property of the coideal $U_{Q,q}^B$. We know our construction and proof work in the setting of~\cite{FanLi,ES} where a Schur-Weyl duality involving the Hecke algebra of type D appears. We expect it to work in many other settings possibly including~\cite{ATY, HuShoji,SaSho,Sho,MazorchukStroppel} where Schur-Weyl dualities appear. The super polynomial functors of Axtell~\cite{Axtell} are also based on the Schur-Weyl dualities of Sergeev~\cite{Sergeev}.

The theory of polynomial functors we develop interacts with type A quantum polynomial functors in two ways. The first interaction is via composition. 

Composition between type A quantum polynomial functors $\mathcal{AP}_{q}^d$ (see Example~\ref{ex:qpf} for the definition) for $q\neq1$ is not possible. See the Introduction to \cite{BuciumasKo1} for a comprehensive discussion explaining this fact. In \cite{BuciumasKo1}, the authors define ``higher degree'' quantum polynomial functors $\mathcal{AP}_q^{d,e}$ (the category $\mathcal{AP}^{d,e}_q$ is denoted in \cite{BuciumasKo1} by $\mathcal P^d_{q,e}$) and define a composition functor $\circ_A : \mathcal{AP}_{q}^{d_1,d_2 e} \times \mathcal{AP}_{q}^{d_2,e} \to \mathcal{AP}_q^{d_1 d_2,e}$. 
The categories $\mathcal{AP}_{q}^{d,e}$ are quantizations of the category of classical polynomial functor $\mathcal{P}^d$ (in the sense of $\mathcal{AP}_{q=1}^{d,e} \simeq \mathcal{P}^d$) but are more complicated: for example we do not know the number of non-isomorphic simple objects in $\mathcal{AP}_{q}^{d,e}$.

In our setting, one cannot hope to define composition of quantum polynomial functors because we cannot take the tensor power of general $U_{Q,q}^B$-modules. In Section~\ref{sec:composition} we define higher degree two-parameter quantum polynomial functors $\mathcal{P}_{Q,q}^{d,e}$ and prove that there is a composition $\circ: \mathcal{P}_{Q,q}^{d_1,d_2 e} \times \mathcal{AP}_{q}^{d_2,e} \to \mathcal{P}_{Q,q}^{d_1 d_2,e}$ that makes the type B higher degree polynomial functors together with type A higher degree polynomials into an action pair. This structure is natural in the setting of polynomial functors while not in the setting of Schur algebra modules. Composition for classical polynomial functors is related to an operation on symmetric polynomials known as plethysm. It would be interesting to understand the analog of plethysm related to our composition between type A and type B quantum polynomial functors (for an introduction to classical plethysm see Macdonald~\cite[Section I.8]{Macdonald}). 

We emphasize that the composition between type A and type B quantum polynomial functors produces what we believe are new, non-trivial examples of action pairs. These examples are different to the examples of the (cylinder braided) action pairs we produce in Section~\ref{sec:cylindertwist}. The latter examples have appeared in a different setting in the work of Kolb and Balagovic and reflect the fact that $U_{Q,q}^B$ is a coideal of $\Uq$.

Higher degree polynomial functors are 
related to certain generalizations of the Schur algebra which we call $e$-Schur algebras and denote by $S_{q}^A(n;d,e)$ and $S_{Q,q}^B(n;d,e)$ (the former was initially defined in~\cite{BuciumasKo1}). They are defined via $e$-Hecke algebras $\mathcal{H}^A_q(d;e)$ and $\mathcal{H}^B_{Q,q}(d;e)$ which live inside the ordinary Hecke algebras $\mathcal{H}^A_q(de)$ and $\mathcal{H}^B_{Q,q}(de)$, respectively; they are higher quantizations of the Weyl groups $W^A_d$ and $W^B_d$, respectively.
See Figure~\ref{matingdragonflies} for the relation between such Schur and Hecke algebras.

The second interaction of type A and type B quantum polynomial functors is presented in Section~\ref{sec:cylindertwist} where we show that the restriction of $\mathcal{AP}_q=\oplus_d \mathcal{AP}^d_q$ to $\mathcal{P}_{Q,q}=\oplus_d \mathcal{P}^d_{Q,q}$ forms a cylinder braided action pair with $\mathcal{AP}_q$. We explain how to generalize this result to higher degree polynomial functors in Remark~\ref{rem:highercylinder}. There also exists a higher degree action of the category $\oplus_d \mathcal{AP}_q^{d,e}$ on $\oplus_d \mathcal P_{Q,q}^{d,e}$ which leads to a new cylinder braided action pair.
The notion of a cylinder braided action pair due to tom Dieck and H\"{a}ring-Oldenburg~\cite{tD, tDHO, HO}, generalizes the notion of a braided monoidal category to a setting where one has categorical solutions of the Yang-Baxter equation and the reflection equation. The quantum symmetric pair $(U_{Q,q}^B, \Uq)$ produces a main example of such a pair. The cylinder braided action pair has an interesting generalization. In \cite[Section 4]{BK} the notion of a braided tensor category with a cylinder twist is developed (Balagovic and Kolb use the term `braided tensor category with a cylinder twist' for what we call cylinder braided action pair); in this generalization, all finite quantum symmetric pairs produce examples of such categories.
A slightly stronger notion than a cylinder braided action pair is that of a braided module category defined in~\cite[\S 4.3]{Enriquez} (see also~\cite[\S~5.1]{Brochier}). Kolb~\cite{KolbBMC}
showed all quantum symmetric pairs for $Q,q$ generic produce such module categories up to twist. Our category of polynomial functors can also be shown to produce braided module categories (see Remark~\ref{rembmc}).



In type A, the tensor power has two distinguished quotients, namely the symmetric power and the exterior power.
In our setting, the two-parameter symmetric power and the exterior power both have two distinguished quotients.
We define them in Section~\ref{sec:pmsymext} and call them the $\pm$-symmetric power, denoted by $S^d_\pm$, and the $\pm$-exterior power, denoted by $\wedge^d_\pm$. They depend on positive and negative eigenvalues of the $K$-matrix, similar to how type A symmetric and exterior power depend on positive and negative eigenvalues of the $R$-matrix.
 These are the most basic examples of the Schur functors and are the building blocks for other Schur functors.
 
In \S~\ref{BZstuff} we define higher degree $\pm$ symmetric and exterior powers. The definition makes crucial use of Corollary~\ref{evKmatrix} where we essentially show that action of the $U_{Q,q}^B$-universal $K$-matrix on any $\Uq$ module has eigenvalues of the form $\pm Q^iq^j$ for $i,j \in \mathbb{Z}$. These examples of higher degree two-parameter quantum polynomial functors should be thought of as the generalization of the type A quantum symmetric and exterior powers due to Berenstein and Zwicknagl~\cite{BZ}.

In Section~\ref{sec:schur}, we construct the Schur functors in $\mathcal{P}_{Q,q}$ analogous to the classical construction of Akin-Buchsbaum-Weyman \cite{ABW}. A classical Schur functor is defined as the image of the conjugation
\[\wedge^{\lambda'}=\wedge^{\lambda'_1}\otimes\cdots\otimes \wedge^{\lambda'_r}\to S^\lambda=S^{\lambda_1}\otimes\cdots\otimes S^{\lambda_l},\]
where $\lambda=(\lambda_1,\cdots,\lambda_r)$ is a partition and $\lambda '=(\lambda'_1,\cdots,\lambda'_l)$ is its transpose.
In our setting, the $\pm$-symmetric/exterior powers defined in Section~\ref{sec:pmsymext} play the role of the symmetric/exterior powers. However, we are unable to define the tensor product of $\pm$-symmetric/exterior powers since they are coideal modules, and not bialgebra modules. Therefore the obvious generalization fails and we need a new idea.     
Our idea is to define a ``deformed tensor product'' of $U_{Q,q}^B$-modules by using the cylinder braided action from Section~\ref{sec:cylindertwist} (an example of deformed tensor products is presented in Definition \ref{defpmpower}) and use it to define the Schur functor. We then write the Schur functor in equation~\eqref{eq:schurexactsequence} generalizing the type A definition of the Schur functor. 
It is defined as the image of a(n induced) conjugation
\[ 
\wedge^{(\lambda',\mu')} \to S^{(\lambda, \mu)},
\] 
where $\wedge^{(\lambda',\mu')}$ is a deformed tensor product of $\wedge_+^{\lambda'_1},\cdots,\wedge_+^{\lambda'_r},\wedge_-^{\mu'_1},\cdots,\wedge_-^{\mu'_l.}$ and $S^{(\lambda, \mu)}$ is similarly a deformed tensor product.
See Definition~\ref{def:schurfunctor} and equation~\eqref{eq:schurexactsequence} for details.

If $Q,q$ are generic, the Schur functors form a complete set of simple objects in the category $\mathcal{P}_{Q,q}$. 
In the non-generic case, we expect that the Schur functors form a complete set of costandard objects whenever $\mathcal{P}_{Q,q}$ is a highest weight category. The latter is true under a small restriction on $Q,q$. 

Our definition of Schur functors can be `lifted' to the setting of higher degree polynomial functors as we explain in \S~\ref{subsec:higherschur}. The result is a class of interesting objects in $\mathcal{P}^{d,e}$ and $\mathcal{AP}^{d,e}$ and is a first step towards understanding the categories $\mathcal{P}^{d,e}$ and $\mathcal{AP}^{d,e}$.

\medskip

\noindent \textbf{Setup.} 
Unless otherwise stated, we assume that $k$ is a field and $Q,q\in k^\times$. 

In a few places, we use the stronger assumption that $k=\mathbb C$ and $Q,q\in k$ are such that $Q^iq^j\neq 1$ for all $i,j\in\mathbb Z$ (in particular $Q,q$ are not roots of unity).
For convenience, we refer to this assumption by saying $Q,q$ are \emph{generic} or by using the term `\emph{generic case}'.
\medskip

\noindent
\textbf{Acknowledgements.} 
We thank Huanchen Bao, Chun-Ju Lai and Catharina Stroppel for useful discussions. We thank Catharina Stroppel for valuable comments on an earlier version of the paper. We thank the referees for many helpful comments. 
Part of the work in this paper was done while the first author visited the Max Planck Institute for Mathematics in Bonn; both authors would like to thank the institute  for hospitality and good working conditions.

Buciumas was supported by ARC grant DP180103150. Ko was supported by the Max Planck Institute for Mathematics in Bonn.

\section{Quantum symmetric pairs and Schur-Weyl dualities}
 We introduce the basic objects which are used throughout the paper: the quantum group $\Uq$, the coideal subalgebra $U_{Q,q}^B$ and the two-parameter Hecke algebra of Coxeter type BC which we denote by $\Hb$. 
We review a Schur-Weyl duality between $\Hb$ and $U_{Q,q}^B$. That is the basis for our definition of two-parameter quantum polynomial functors.

\subsection{Hecke algebras}\label{subsec:hecke}

\subsubsection{Definition}
Denote the Weyl group of type BC of rank $d$ by $W^B(d)$. It is the Coxeter group with generators $s_i, 0 \leq i \leq d-1$ and relations

\begin{equation*}
\begin{aligned}[c]
&s_i^{2}=1\\
&s_i s_{i+1} s_i = s_{i+1} s_i s_{i+1} \\
&s_0s_1s_0s_1=s_1s_0s_1s_0, \\
&s_i s_j=s_j s_i \\
\end{aligned}
\hspace{1cm}
\begin{aligned}[c]
&\textnormal{ for } i \geq 0,\\
&\textnormal{ for } i > 0,\\
&    \\
&\textnormal{ for } |i-j| > 1.
\end{aligned}
\end{equation*}
The elements $s_i\in W^B(d)$ for $i>0$ generate a subgroup isomorphic to $W^A(d)$, the Weyl group of type A (otherwise known as the symmetric group $S_d$).

Let $\mathcal{H}_{Q,q}^{B}(d)$ be the two-parameter Hecke algebra of type BC~\cite{LusztigHeckebook}. It is presented by generators $T_0, T_1,\cdots, T_{d-1}$ satisfying the relations 
\begin{equation}\label{HeckedefinitionBC}
\begin{aligned}[c]
&(T_0+Q)(T_0-Q^{-1})=0,\\
&(T_i+q)(T_i-q^{-1})=0\\
&T_i T_{i+1} T_i = T_{i+1} T_i T_{i+1} \\
&T_0T_1T_0T_1 = T_1 T_0 T_1 T_0, \\
&T_iT_j = T_j T_i \\
\end{aligned}
\hspace{1cm}
\begin{aligned}[c]
&\\
&\textnormal{ for } i > 0,\\
&\textnormal{ for } i > 0,\\
&    \\
&\textnormal{ for } |i-j| > 1.
\end{aligned}
\end{equation}
Note that the generators $T_1, \cdots, T_{d-1}$ generate a subalgebra of $\mathcal{H}_{Q,q}^{B}(d)$ isomorphic to the Hecke algebra $\mathcal{H}_q^A(d)$ of type A. 

Given an element $w\in W^B(d)$, we write $T_w=T_{i_1}\cdots T_{i_l}$ where $s_{i_1}\cdots s_{i_l}$ is a reduced expression of $w$. The element $T_w\in \mathcal{H}_{Q,q}^{B}(d)$ does not depend on the reduced expression. The elements $T_w$ for $w \in W^B(d)$ form a basis of $\mathcal{H}_{Q,q}^B(d)$.

\subsubsection{Action on the tensor space}\label{sssec:actionontensor}

Set $n=2r+1$ or $2r$ and denote $\mathbb{I}:=\mathbb{I}_n$. 
If $n$ is even we define $\mathbb{I}_{2r}:=\{-\frac{2r-1}{2}, \cdots, -\frac{1}{2}, \frac{1}{2}, \cdots ,\frac{2r-1}{2}\}$ and otherwise we let $\mathbb{I}_{2r+1}:= \{-r, \cdots ,-1, 0, 1, \cdots, r\}$. 

Let $\mathbf{a}:=(a_1,\cdots, a_d) \in \mathbb{I}^d$. 
The group $W^B_d$ acts on the set $\mathbb{I}^d$ as follows~\cite{BW13, ES}:
\begin{equation}\label{WeylgroupactionBC}
\begin{split}
&s_i: (\cdots, a_i,a_{i+1},\cdots) \mapsto (\cdots, a_{i+1},a_{i},\cdots)   \textnormal{ for } i > 0, \\
&s_0:(a_1,\cdots) \mapsto (-a_1, \cdots).
\end{split}
\end{equation}

Let $V_n$ be a vector space with basis $\{ v_i, i \in \mathbb{I}_n \}$. Write $v_\mathbf{a}:= v_{a_1} \otimes \cdots \otimes v_{a_d} \in V_n\td$. Then the set $\{v_{\mathbf{a}}, \mathbf{a}\in \mathbb{I}^d\}$ is a basis for $V_n^{\otimes d}$. 

There is a right action of $\mathcal{H}_{Q,q}^{B}(d)$ on $V_n^{\otimes d}$ given by 
\begin{equation}\label{HeckeactionBC}
\begin{split}
&T_i \mapsto (R_q)_{i,i+1} \textnormal{ for } i > 0, \\
&T_0 \mapsto (K_Q)_1,
\end{split}
\end{equation}
where $R_q:V_n \otimes V_n \to V_n \otimes V_n$ is the map 
\begin{equation}\label{eq:Rmatrixdef} 
R_q: v_i \otimes v_j \mapsto \begin{cases} 
      q^{-1} v_i \otimes v_j & \textnormal{ if } i=j, \\
      v_j \otimes v_i & \textnormal{ if } i<j, \\
      v_j \otimes v_i + (q^{-1}-q)v_i \otimes v_j & \textnormal{ if } i>j,
   \end{cases}
\end{equation}
and $K_Q : V_n \to V_n$ is the map
\begin{equation}\label{eq:Kmatrixdef}
K_Q: v_i \mapsto \begin{cases} 
      Q^{-1} v_i  & \textnormal{ if } i=0, \\
      v_{-i}  & \textnormal{ if } i>0 , \\
      v_{-i} + (Q^{-1}-Q)v_i  & \textnormal{ if } i<0. 
   \end{cases}
\end{equation}
The map $(R_q)_{i,i+1}$ acts as $R_q$ on the $(i,i+1)$ entries of the tensor product $V_n^{\otimes d}$ and as the identity on the rest of the entries. Similarly, $(K_Q)_1=K_Q\otimes \id_{V_n}^{\otimes d-1}$.
The action of $\mathcal{H}_{Q,q}^B(d)$ is classical. See for example Green \cite{Greenhyper}. The Schur algebra $S_{Q,q}^B(n;d)$ is then defined as the centralizer algebra of the right action of $\mathcal{H}_{Q,q}^B(d)$ on the tensor space $V_n\td$. 

\begin{rem}
The map $R_q$ is the action of the inverse of the universal $R$-matrix of $\Uq$ on $V_n \otimes V_n$ as explained in \cite[Proposition 5.1]{BW13} in the $Q=q$ case. Similarly, the map $K_q$ is the action of the inverse of the universal $K$-matrix (due to \cite{BK}) of the coideal $U_{Q,q}^B$ on $V_n$ (see \cite[Theorem 5.4, Theorem 6.27]{BW13}, again for the $Q=q$ case). 
\end{rem}

\subsubsection{The elements $K_i$}

For each $1\leq i\leq d$, we consider the elements \[K_i=T_{i-1}\cdots T_1T_0T_1\cdots T_{i-1}\]
in $\HB$. These are the Jucy-Murphy elements of $\Hb$ (see~\cite[Section 2]{DipperJamesMathas}). The following lemma is well-known, see for example~\cite[Proposition 2.1]{DipperJamesMathas} for a proof. 

\begin{lem}\label{lem:kicommute}
For each $1\leq i,j\leq d$, $K_i$ and $K_j$ commute.
\end{lem}

Let $c_K \in \HB$ be the element 
\begin{equation}\label{eqxK}
 c_K := \prod_{i=1}^d K_i. 
\end{equation}
The product is well-defined due to Lemma \ref{lem:kicommute}.

\begin{lem}\label{cKcentral}
The element $c_K\in \HB$ is central.
\end{lem}
\begin{proof}
We show that $c_K$ commutes with all the generators $T_i$ of $\HB$. 

First let us look at $T_0$. It obviously commutes with itself. It commutes with $T_1T_0T_1$, this is just the equation $T_1T_0T_1T_0 = T_0 T_1 T_0T_1$. It also commutes with $T_i, i>1$. This means it commutes with $T_j\cdots T_2 (T_1 T_0 T_1) T_2 \cdots T_j$. Therefore it commutes with $c_K$. 

Let us look at $T_i$ for $i>0$. The following facts are parts ii) and iii) of~\cite[Proposition 2.1]{DipperJamesMathas}:
\begin{enumerate}
\item $T_i$ commutes with $K_{j}$ for $i \neq j, j-1$. 
\item $T_i$ commutes with $K_{i+1} K_i$. 
\end{enumerate}
We conclude that $T_i$ commutes with $c_K$. 
\end{proof}


Consider the action of $\HB$ on $V_n\td$ defined in \S\ref{sssec:actionontensor}. We close the section by determining the eigenvalues of $K_i$. The following lemma \cite[Lemma 5.2]{MaksimauStroppel} comes useful.

\begin{lem}\label{newlemma}
Suppose $K_i$, $K_{i+1}$ has a simultaneous eigenvector with eigenvalues $a,b$ (respectively). Then either $K_i,K_{i+1}$ also has a simultaneous eigenvector with eigenvalues $b,a$ or $b=q^{\pm2}a$.
\end{lem}
\begin{proof}
Let $v\in V_n\td$ be a simultaneous eigenvector for $K_i$, $K_{i+1}$ with eigenvalues $a,b$ (respectively). Then the vector $w=(q\inv-q)bv+(a-b) T_iv$ is checked to satisfy $K_iw=bw$ and $K_{i+1}=aw$. 
If $w\neq 0$, then $w$ is a desired eigenvector.
If $w=0$ then $v$ is an eigenvector for $T_i$. This implies $K_{i+1}v=T_iK_iT_iv=ac^2v$ where $c$ is an eigenvalue for $T_i$, which is either of $-q$ or $q\inv$.
\end{proof}

\begin{prop}\label{eigenvaluesKi}
The eigenvalues of $K_i$ on $V_n\td$ are of the form $-Qq^{2j}$ and $Q\inv q^{2j}$ where $|j|<i$. 
\end{prop}
\begin{proof}
The $i=1$ case follows from the definition (and also follows from the the relation $(T_0-Q\inv)(T_0+Q)=0$ in the Hecke algebra). 

Now suppose that the eigenvalues of $K_i$ are of the form $-Qq^{2j}$ and $Q\inv q^{2j}$ where $|j|<i$, and let $b$ be an eigenvalue of $K_{i+1}$. The actions of $K_i$ and $K_{i+1}$ are simultaneously triangularizable, so we can find a simultaneous eigenvector $v$ for $K_i,K_{i+1}$ where $K_{i+1}v=bv$. Then by Lemma~\ref{newlemma}, either $b=q^{\pm 2}a$ where $a$ is an eigenvalue of $K_i$ (the second case of the lemma) or  $b$ is an eigenvalue of $K_i$  (the first case of the lemma). Therefore $b$ should be of the desired form.
\end{proof}

\begin{cor}\label{evKmatrix}
The eigenvalues of $c_K$ are of the form $\pm Q^iq^j$ for $i,j\in\mathbb Z$.
\end{cor}
\begin{proof}
Since $K_i$ are simultaneously triangularizable, each eigenvalue of $c_K$ is a product of eigenvalues of $K_i$'s. The claim thus follows from Proposition~\ref{eigenvaluesKi}
\end{proof}

\subsection{Coideal subalgebras and Schur algebras}\label{coidealschurdef}

\subsubsection{Schur algebras}\label{subsec:schur}
Considering the action  of $\mathcal{H}_{Q,q}^B(d)$ on $V_n\td$ in equation~\eqref{HeckeactionBC}, define
\begin{equation}\label{def:schuralgebraB}
S^B_{Q,q}(m,n;d):=  \Hom_{\mathcal{H}_{Q,q}^B} (V_m^{\otimes d}, V_n^{\otimes d}).  
\end{equation}
Then the Schur algebra $S^B_{Q,q}(n;d)$ is the specialization of $S^B_{Q,q}(m,n;d)$ at $m=n$; it is an algebra with multiplication given by composition and the identity given by the identity homomorphism.   

There is an obvious action $S_{Q,q}^B(n;d) \rotatebox[origin=c]{-90}{$\circlearrowright$} V_n\td$. 

\subsubsection{Quantum groups and coideal subalgebras}\label{sec:defcoideal}
In this subsection, we assume that $k=\mathbb C$ and $Q,q$ are generic.\footnote{The reason we need this assumption is that the coideal subalgebra $U^B_{Q,q}$ is defined and studied only when $q,Q$ are generic (to the authors' knowledge at the point when this work is written). 
When $q$ or $Q$ is a root of unity, we expect there to be a definition of the coideal $U_{Q,q}^B$ similar to Lusztig's quantum group at a root of unity \cite{Lusztigrou}, which still surjects to the Schur algebra $S_{Q,q}^B$.
}

The quantum group $U_q(\mathfrak{gl}_n)$ is the unital associative algebra over $\mathbb{C}$ generated by elements $E_i, F_i$ for $i \in \mathbb{I}_{n-1}$ and $D_i^{\pm}$ for $i \in \mathbb{I}_n$ subject to the relations (set $j' = j-\frac{1}{2}$):
\begin{equation}\label{qgrelations}
\begin{split}
&     D_i D_j = D_j D_i, \quad D_i D_i^{-1}=1=D_i^{-1} D_i, \\
&      D_i E_j D_i^{-1} = q^{\delta_{i,j'} - \delta_{i-1, j'} } E_j, \quad D_i F_j D_i^{-1}= q^{-\delta_{i,j'} + \delta_{i-1, j'} }  F_j, \\[0.5em]
&    E_i E_j = E_j E_i, \quad F_i F_j = F_j F_i  \quad \text{ if } i \neq j \pm 1, \\
&    E_i F_j - F_j E_i = \delta_{i,j}\frac{D_{j'} D_{j'+1}^{-1} - D_{j'+1}D\inv_{j'}}{q-q^{-1}}, \\
&    E_i^{2} E_{i \pm 1} - (q+q^{-1})E_i E_{i\pm 1} E_i  + E_{i\pm 1}E^2_i  = 0, \\
&    F_i^{2} F_{i \pm 1} - (q+q^{-1})F_i F_{i\pm 1} F_i  + F_{i\pm 1}F^2_i = 0.
\end{split}
\end{equation}
Let $H_j = D_{j'} D_{j'+1}^{-1}$. The subalgebra of $U_q(\mathfrak{gl}_n)$ generated by $E_i, F_i, H_i $ for $i \in \mathbb{I}_{n-1}$ is the quantum group $U_q(\mathfrak{sl}_n)$. We do not define the quantum group at a root of unity, but whenever we mention it, we are referring to Lusztig's version of the quantum group at a root of unity \cite{Lusztigrou}.

The quantum group $U_q(\mathfrak{gl}_n)$ is a Hopf algebra with comultiplication $\Delta$ and antipode $S$ given on generators by the following formulas:
\begin{equation}\label{qgcomultiplication}
    \begin{split}
        \Delta(D_i) &= D_i \otimes D_i, \\
        \Delta(E_i) &= 1 \otimes E_i + E_i \otimes H_i\inv, \\
        \Delta(F_i) &= F_i \otimes 1 + H_i \otimes F_i, \\
        S(D_i) &= D_i^{-1}, \,\,\, S(E_i)=-E_i H_i,\,\,\, S(F_i)= -H_i\inv F_i.
    \end{split}
\end{equation}
Let $V_n$ be the defining representation of $\Uq$ described in \S2.1.2; it has basis $\{v_i, i \in \mathbb{I}_n\}$ and the quantum group $U_q(\mathfrak{gl}_n)$ acts on $V_n$ as follows:
\begin{equation}\label{qgaction}
\begin{split}
D_i v_j &= q^{\delta_{i,j}} v_j, \\
E_i v_j &= \delta_{i, j'} v_{j-1} , \\
F_i v_j &= \delta_{i,j'+1} v_{j+1}.
\end{split}
\end{equation}

We now introduce the (right) coideal subalgebra $U^B_{Q,q} (\mathfrak{gl}_n)$ as in \cite{BWW}, where it is denoted by $\mathbf{U}^i$ or $\mathbf{U}^j$, depending on the parity of $n$. For $i \in \mathbb{I}_{n-1}, j \in \mathbb{I}_n$ define the following elements of $\Uq$:

\begin{equation}\label{eq:coidealgen}
\begin{split}
 d_i = D_i D_{-i}, \,\,\, e_i &= E_i +F_{-i}H_{i}^{-1}, \,\,\, f_{i} = E_{-i} + H_{-i}^{-1}F_{i}, \\ 
e_{\frac{1}{2}} = E_{\frac{1}{2}} + Q^{-1}F_{-\frac{1}{2}} H_{\frac{1}{2}}^{-1}, \,\,\, f_{\frac{1}{2}} &= E_{-\frac{1}{2}} + Q H^{-1}_{-\frac{1}{2}} F_{\frac{1}{2}}, \,\,\, t = E_0 + q F_0 H_0^{-1} + \frac{Q-Q^{-1}}{q-q^{-1}}H_0^{-1}. 
\end{split}
\end{equation}

The subalgebra $U^B_{Q,q} (\mathfrak{gl}_n)$ of $\Uq$ is generated by the elements $e_i, f_i$ for $i \in \mathbb{I}_{n-1}, i>0$ , $d_i$ for $i\in \mathbb I_{n}, i>0$, and the element $t$ when $n$ is odd. We denote $U^B_{Q,q} (\mathfrak{gl}_n)$ by $ U_{Q,q}^B$ throughout the text. The name coideal subalgebra is due to the fact that the restriction of the comultiplication from $\Uq$ to $U_{Q,q}^B$ has image in $U_{Q,q}^B \otimes \Uq.$
The $\Uq$-module $V_n\td$ restricts to an $U_{Q,q}^B$-module. Then the left action of $U_{Q,q}^B$ and the right action of $\Hb$ on $V_n\td$ commute. 
Moreover, we have

\begin{thm}\label{thm:doublecentralizerB}
 \cite[Theorem 2.6, Theorem 4.4]{BWW} The actions of $U_{Q,q}^B$ and $\Hb$ on $V_n\td$ form double centralizers. 
\end{thm}

\begin{rem}\label{coidealschur}
By Theorem \ref{thm:doublecentralizerB} one realizes the Schur algebra $S_{Q,q}^B(n;d)$ as a quotient of the coideal subalgebra $U_{Q,q}^B$. 
This gives an equivalence of categories between the category of degree $d$ modules of $U_{Q,q}^B$ (i.e. summands of $V_n\td$) and the category of $S_{Q,q}(n;d)$-modules. 
Our main results in Section \ref{sec:quantumpolfunctors} identifies degree $d$ polynomial functors with representations of the Schur algebra $S_{Q,q}^B(n;d)$ for $n \geq d$. The fact that the category of finite dimensional representations of $S_{Q,q}^B(n;d)$ is equivalent to the same category as long as $n\geq d$
can be interpreted as a stability result in the limit $n\to\infty$ for $U_{Q,q}^B$ when $Q$ and $q$ are generic. This is different to the $d\to\infty$ stabilization studied in \cite{BKLW}.
\end{rem}

For a partition $\lambda$, let $|\lambda|$ be the sum of its parts and $\ell(\lambda)$ the number of non-zero entries in $\lambda$. Under our assumption,
the algebra $\mathcal{H}_{Q,q}^{B}(d)$ is semisimple and has irreducible representations $M_{\lambda, \mu}$ indexed by pairs of partitions $(\lambda, \mu)$ with $|\lambda| + |\mu| =d$ (this follows from the work of \cite{DJ92}). Furthermore, there is a $\SB\otimes \HB$-bimodule decomposition of $V_n\td$ (note that using
Theorem~\ref{thm:doublecentralizerB} we can view it as a decomposition as a $U_{Q,q}^B \otimes \Hb$-bimodule):

\begin{equation}\label{eq:SWdecompositionB}
V_n^{\otimes d}  \cong \bigoplus_{(\lambda,\mu)\vdash_n d} L_{\lambda,\mu}(n) \otimes M_{\lambda, \mu}.  
\end{equation}
The subscript $(\lambda,\mu) \vdash_n d$ means that $\lambda, \mu$ are partitions such that $|\lambda|+|\mu|=d$ and $\ell(\lambda)\leq r, \ell(\mu) \leq r$ when $n=2r$ or $\ell(\lambda)\leq r+1, \ell(\mu) \leq r$ when $n=2r+1$. In the above, $L_{\lambda, \mu}(n)$ is either an irreducible representation of $U_{Q,q}^B$ or $0$. If $n\geq 2d$, $L_{\lambda, \mu}(n)$ is never $0$. These irreducibles are indexed by bipartitions $(\lambda, \mu) \vdash_n d$.

A useful consequence of \eqref{eq:SWdecompositionB} is the following fact.

\begin{prop}\label{Kidiag}
The $K_i$ action on $V\td$ is diagonalizable.
\end{prop}
\begin{proof}
We first show that the element $c_K=\prod_{i=1}^d K_i$ is diagonalizable.
The element $c_K$ is central in $\HB$ by Lemma \ref{cKcentral}. It further commutes with the action of $\SB(n;d)$, so it is a central ($\SB(n;d),\HB$)-bimodule action of $V\td$ (if we view $(\SB(n;d),\HB)$-bimodule as a left $\SB(n;d)\otimes \HB^{op}$-module, then $c_K$ is in the center of $\SB(n;d)\otimes \HB^{op}$). 
Since the decomposition is multiplicity free, $c_K$ acts by a scalar on each irreducible bimodule summand of $V\td$, hence diagonal on $V\td$. 

Now we proceed by induction on $d$. We know that $K_1$ is diagonalizable, which takes care of the $d=1$ case. Let $d>1$. By induction hypoethesis, for each $i<d$, $K_i$ is diagonalizable. (In fact, the induction hypothesis says that $K_i$ is diagonalizable on $V^{\otimes i}$, but then $K_i|_{V\td}=K_i|_{V^{\otimes i}}\otimes \id^{\otimes d-i}$ is also diagonalizable.) Writing $K_d=c_K K_{d-1}\inv \cdots K_1\inv$, we see that $K_d$ is a product of diagonalizable elements. By Lemma \ref{lem:kicommute} and Lemma \ref{cKcentral}, the elements all commute and hence are simultaneously diagonalizable. This implies that $K_d$ is diagonalizable. 
\end{proof}

\begin{rem}
The Schur algebra defined above is a generalization of the type A $q$-Schur algebra of Dipper and James \cite{DipperJames}. It has first appeared in~\cite{Greenhyper} and it is the same Schur algebra appearing in~\cite{BWW} or in~\cite{LNX}.
It is different to the Cartan type B generalization defined in terms of the vector representation of the type B quantum group and the BMW algebra.
\end{rem}

\subsection{Young symmetrizers for $\mathcal{H}_{Q,q}^B(d)$}\label{subsec:youngsym}

In this subsection, we assume $k=\mathbb C$ and $Q,q$ are generic. 
We explain the construction of certain Young symmetrizers for the Hecke algebra $\mathcal{H}_{Q,q}^B(d)$ following Dipper and James \cite{DJ92}. We then describe irreducible representations of $U_{Q,q}^B$ as images of these Young symmetrizers acting on $V_n\td$ by Schur-Weyl duality in Theorem~\ref{thm:doublecentralizerB}.  

Consider the following elements $u_i^{\pm} \in \mathcal{H}_{Q,q}^B(d)$:
\begin{equation}\label{eq:u}
u_i^+ = \prod_{j=1}^{i} (K_j + Q), \,\,\,\,\,\, u_i^- = \prod_{j=1}^{i} (K_j - Q^{-1}). 
\end{equation}

Given $a$ and $b$ non-negative integers, define $w_{a,b} \in W^A(d) \subset W^B(d)$ to be the element given in two line notation by

\begin{equation}\label{eq:wab}
      w_{a,b} = \bigl(\begin{smallmatrix}
    1 & \cdots & b & b+1 & \cdots & a+b \\
    a+1 & \cdots & a+b & 1 &  \cdots  & a
  \end{smallmatrix}\bigr).
\end{equation}
 
Let $T_{a,b}:=T_{w_{a,b}}$ be the corresponding element in $\mathcal{H}_{Q,q}^B(d)$. 
Let $\tilde{z}_{b,a}$ be the element defined in \cite[Definition 3.24]{DJ92}. Note that by definition $\tilde{z}_{b,a}$ is a central element of $\mathcal{H}_q(S_a \times S_b) \subset \mathcal{H}^A_q(a+b) \subset \mathcal{H}_{Q,q}^B(a+b)$, where we define $\mathcal{H}_q(S_a\times S_b)$ as the subalgebra of $\mathcal{H}_q^A(a+b)$ with generators $T_i, i \neq a$. The element $\tilde{z}_{b,a}$ satisfies 
\[ u_a^+ T_{a,b}u_b^- T_{b,a} u_a^+ T_{a,b}u_b^- = \tilde{z}_{b,a} u_a^+ T_{a,b}u_b^-.  \]
and it is invertible by \cite[\S 4.12]{DJ92}.
Finally define the following element as in \cite[Definition 3.27]{DJ92}:
\begin{equation}\label{eq:eab}
e_{a,b} = T_{a,b} u_b^{-1} T_{b,a} u_a^+ \tilde{z}^{-1}_{b,a}     
= \tilde{z}^{-1}_{b,a} T_{b,a}u_b^{-1} T_{b,a} u_a^+.
\end{equation}
Then $e_{a,b}$ commutes with all elements in $\mathcal{H}_q(S_a\times S_b)$. The following are proved in \cite{DJ92} under the assumption that the element
\begin{equation}\label{DJpolynomial}
f_d(Q,q) =\prod_{i=1-d}^{d-1} (Q^{-2}+q^{2i}) 
\end{equation}
is nonzero, which is covered under our assumption.
\begin{thm}\label{DJthm}
Let $a,b$ be non-negative integers such that $a+b=d$. Then 
\begin{enumerate}
\item $e_{a,b} \mathcal{H}_{Q,q}^B(d) e_{a,b} = e_{a,b}\mathcal{H}_q(S_a \times S_b) \simeq \mathcal{H}_q(S_a \times S_b)$. 
\item There is a Morita equivalence
\[ \mathcal{H}_{Q,q}^B(d) \simeq \oplus_{i=0}^d e_{i,d-i} \mathcal{H}_{Q,q}^B(d) e_{i,d-i}.  \]
\end{enumerate}
\end{thm}

Let $e^a_\lambda\in \mathcal{H}_q^A(a)$ be the (type A) quantum Young symmetrizers (see Gyoja \cite{Gyoja} for a definition).
Since $q$ is generic, the algebra $\mathcal{H}_q(S_a \times S_b)=\mathcal{H}_q(S_a)\times \mathcal{H}_q(S_b)$ is semisimple, and the set $\{\mathcal{H}_q(S_a \times S_b)e^a_\lambda e^b_\mu\ |\ \lambda \vdash a, \mu \vdash b\}$ gives a complete list of isomorphism classes for irreducible $\mathcal{H}_q(S_a \times S_b)$-modules.
Now let
\[e_{\lambda,\mu}:=e_{a,b} e^a_\lambda e^b_\mu = e^a_\lambda e^b_\mu e_{a,b}.\] 
Then it follows from Theorem \ref{DJthm} that $\{\HB e_{\lambda,\mu}\ |\ (\lambda,\mu)\vdash d\}$ forms a complete list of non-isomorphic irreducible modules for $\HB$.

Now we apply the Schur-Weyl duality to construct all the irreducible polynomial $U^B_{Q,q}$-modules up to isomorphism.

\begin{prop}\label{egivesirred}
The image in $V_n\td$ of the action of $e_{\lambda,\mu}\in \mathcal{H}_{Q,q}^B(d)$ is isomorphic to $L^B_{\lambda,\mu}(n)$. 
\end{prop}
\begin{proof}
This follows from the bimodule decomposition \eqref{eq:SWdecompositionB} of $V_n\td$. That is,
\begin{equation*}
    \begin{split}
        V_n\td e_{\lambda,\mu}\cong V_n\td\otimes_{\HB} \HB e_{\lambda,\mu}&\cong \bigoplus_{\lambda',\mu'} L_{\lambda',\mu'}(n)\otimes M_{\lambda',\mu'}\otimes_{\HB} \HB e_{\lambda,\mu}\\
        &\cong\bigoplus_{\lambda',\mu'} L_{\lambda',\mu'}(n)\otimes M_{\lambda',\mu'}\otimes_{\HB} M_{\lambda,\mu}\\
        &\cong\bigoplus_{\lambda',\mu'} L_{\lambda',\mu'}(n)\otimes \delta_{(\lambda,\mu),(\lambda',\,u')}k\\        
        &\cong L_{\lambda,\mu}(n).
    \end{split}
\end{equation*}
In the second from the last isomorphism, we use that $\HB$ is a symmetric algebra (see~\cite[Section 5]{CIK}).
\end{proof}
There is no explicit formula for $\tilde{z}_{b,a}$ and therefore the element $e_{a,b}$ is not useful when performing explicit computations. We can bypass this difficulty by working with the following element: 
\begin{equation}\label{defe'}
e'_{\lambda,\mu}:= T_{a,b} u_b^{-1} T_{b,a} u_a^+   e^a_\lambda e^b_\mu=e_{\lambda,\mu} \tilde{z}_{b,a}.
\end{equation}

\begin{prop}\label{e'givesirred}
The image in $V_n\td$ of the action of $e'_{\lambda,\mu}\in\mathcal{H}_{Q,q}^B(d)$ is isomorphic to $L_{\lambda,\mu}(n)$. 
\end{prop}

\begin{proof}
By Proposition \ref{egivesirred}, it is enough to show that $V_n\td e_{\lambda,\mu}$ is isomorphic to $V_n\td e'_{\lambda,\mu}$.
Consider the map \[m:  V_n\td e_{\lambda,\mu} \to V_n\td e'_{\lambda,\mu}= V_n\td e_{\lambda,\mu} \tilde{z}_{b,a}\] given by the (right) action of $\tilde{z}_{b,a}\in \HB$ on $V_n\td e_{\lambda,\mu}$.
Since the $U_{Q,q}^B$ action on $V_n\td e_{\lambda,\mu}$ commutes with the $\HB$ action, the map $m$ is an $U_{Q,q}^B$-morphism.
Since $\tilde{z}_{b,a}$ is invertible, the map $m$ is an $U_{Q,q}^B$-isomorphism.
\end{proof}

The elements $e'_{\lambda,\mu}$ are not (quasi-)idempotents, but we still call them Young symmetrizers.

\subsection{Permutation modules for Hecke algebras}\label{subsec:permutation}

Given $\mathbf a \in \mathbb I_n^d$, the subspace $V(\mathbf a)$ of $V_n\td$ spanned by \{$v_{\sigma\mathbf a}\ |\ \sigma\in W^B \}$ is invariant under the action of $\mathcal{H}_{Q,q}^B(d)$. Sometimes we write $V(\mathbf a, n)$ to clarify where $\mathbf a$ belongs. Thus, we have a decomposition
\begin{equation}\label{eq:tensordecompositionheclemodule} 
V_n\td=\bigoplus_{\mathbf a\in \mathbb I_n^d/W^B(d)} V(\mathbf a,n) 
\end{equation}
as $\mathcal{H}_{Q,q}^B(d)$-modules. 

Alternatively, we can index the permutation modules by compositions of $d$. Let $\theta := \left(\theta_{\frac{-n+1}{2}}, \cdots, \theta_{\frac{n-1}{2}}\right)$ be a composition of $d$. Define $\mathbf{a}(\theta)$ via the following equation: 
\begin{equation}\label{aintermsoftheta}
v_{\mathbf{a}(\theta)} :=  \bigotimes_{j=-(n-1)/2}^{(n-1)/2} v_{j}^{\otimes \theta_{j}}.
\end{equation}
Let $V_{\theta}:=V(\mathbf{a})$ be the subspace of $V_n^{\otimes d}$ spanned by $v_{s(\mathbf{a})}, s \in W^B(d)$. Then $V_\theta$ is a direct summand of $V_n^{\otimes d}$ (as an $\mathcal{H}_{Q,q}^B(d)$-module). Moreso, $V_n^{\otimes d}$ is a direct sum of objects isomorphic to $V_\theta$ for certain $\theta$. 

Adding $0$'s in pairs at a place $j>0$ to a composition $\theta = \left(\theta_{\frac{-n+1}{2}}, \cdots, \theta_{\frac{n-1}{2}}\right)$ means defining a new composition $\theta ' = \left(\theta_{\frac{-n-1}{2}} ', \cdots, \theta_{\frac{n+1}{2}} ' \right)$ such that:
\[ \theta_{l} ':= \begin{cases} 
      \theta_l & \textnormal{ if } -j<l <j, \\
      0  & \textnormal{ if } l = \pm j , \\
      \theta_{l\pm 1}  & \textnormal{ if } l \gtrless \pm j. 
      \end{cases}    \]
 For example, adding $0$'s at $j=1$ to $\theta = (2,1,2)$ produces $\theta ' = (2,0,1,0,2)$. If $V_\theta \subset V_{n}^{\otimes d}$, then clearly $V_{\theta'} \subset V_{n+2}^{\otimes d}$.      
      
Adding a $0$ at $j=0$ to a composition $\theta$ as above for $n$ even means defining a new composition $\theta ' = \left(\theta_{\frac{-n}{2}} '', \cdots, \theta_{\frac{n}{2}} '' \right)$ such that:
\[ \theta_{l} '':= \begin{cases} 
      0  & \textnormal{ if } l = 0 , \\
      \theta_{l\mp \frac{1}{2}}  & \textnormal{ if } l \gtrless 0. 
      \end{cases}    \]
For example adding a $0$ at $j=0$ to $\theta = (1,2,3,4)$ produces $\theta '' = (1,2,0,3,4)$. If $V_\theta \subset V_{n}^{\otimes d}$, then $V_{\theta ''} \subset V_{n+1}^{\otimes d}$.       
      
There is an obvious inverse procedure to adding $0$'s in pairs at a place $j>0$ if $\theta_{\pm j}=0$ (and similarly there is an inverse procedure for adding a $0$ at $j=0$ when $\theta_0=0$).

\begin{lem}\label{lem:permutationmodulesiso}
The $\mathcal{H}_{Q,q}^B(d)$-modules $V_\theta$, $V_{\theta '}$ and $V_{\theta ''}$ are isomorphic. 
\end{lem}
\begin{proof}
Let us explain the isomorphism between $V_\theta$ and $V_{\theta '}$ since the case $V_{\theta ''}$ is similar.

The space $V_\theta$ is spanned as an $\mathcal{H}_{Q,q}^B(d)$-module by the vector $v_{\mathbf{a}}$, for $\mathbf{a}$ given in terms of $\theta$ by equation~\eqref{aintermsoftheta}, while the space  $V_{\theta '}$ is spanned by vector $v_{\mathbf{a}'}$ for $\mathbf{a}'$ given in terms of $\theta '$ by equation~\eqref{aintermsoftheta}. There is a unique vector space isomorphism between $V_\theta$ and $V_{\theta '}$ that maps $v_{s \mathbf{a} } \mapsto v_{s \mathbf{a} '}$ for all $s \in W^B(d)$. Because of the way the vector space isomorphism is defined (i.e. it is essentially defined on pure tensors by replacing $v_i / v_{-i}$ by $v_{i+1} / v_{-i-1}$ for all $i>j$), this map commutes with the action of $T_i$ defined in~\eqref{HeckeactionBC} and therefore is an isomorphism of $\mathcal{H}_{Q,q}^B(d)$-modules. 

For example, if $\theta = (2,1,3)$ and $\theta ' = (2,0,1,0,3)$, then $v_{\mathbf{a}} = v_1\otimes v_1\otimes v_0 \otimes v_{-1} \otimes v_{-1} \otimes v_{-1}$ and $v_{\mathbf{a} '} = v_2\otimes v_2\otimes v_0 \otimes v_{-2} \otimes v_{-2} \otimes v_{-2}$. The isomorphism between $V_\theta$ and $V_{\theta '}$ maps, for example, $v_1\otimes v_0\otimes v_{1} \otimes v_{1} \otimes v_{-1} \otimes v_{-1} \mapsto v_2\otimes v_0\otimes v_{2} \otimes v_{2} \otimes v_{-2} \otimes v_{-2}$.
\end{proof}

In terms of $\mathbf a\in \mathbb I_n^d$, we get the following stability lemma.

\begin{lem}\label{permut}\label{lemma:generationodd}
Let $r\geq d$. Then for any $n$ and $\mathbf a\in \mathbb I_n^d$, the $\mathcal{H}_{Q,q}^B(d)$-module $V(\mathbf a,n)$ is isomorphic to $V(\mathbf b,2r+1)$ for some $\mathbf b\in \mathbb I_{2r+1}^d$.
\end{lem}
\begin{proof}
The result follows by use of Lemma \ref{lem:permutationmodulesiso}. Let $\theta(\mathbf{a})$ be the composition associated to $\mathbf{a}$ and let $\theta(\mathbf{b})$ be the composition associated to $\mathbf{b}$. If $n$ is odd and less than or equal to $2r+1$, we can add $0$'s in pairs to $\theta(\mathbf{a})$ to obtain a $\theta (\mathbf{b})$ such that $V(\mathbf a,n) \simeq V_{\theta(\mathbf{a})} \simeq V_{\theta(\mathbf{b})} \simeq V(\mathbf b,2r+1)$. If $n$ is larger than $2r+1$ then $n$ is larger than $2d+1$ and the composition $\theta(\mathbf{a})$ has at most $d$ non-zero entries. Therefore we can subtract $0$'s in pairs from $\theta (\mathbf{a})$ to obtain a $\theta (\mathbf{b})$ with the required properties. 

If $n$ is even, we first add a $0$ at $j=0$ to the composition associated to $\mathbf{a}$ and then follow the same procedure as in the odd $n$ case. 
\end{proof}

\subsection{Generalized Schur algebras and $e$-Hecke algebras}\label{sec:RKe} 

The category of polynomial representations of $\Uq$ is a braided monoidal category. That is, given polynomial $\Uq$-modules $V$ and $W$, there is a $\Uq$-module isomorphism $R_{V,W}:V \otimes W \to W \otimes V$ that satisfies the Yang-Baxter equation:
 \begin{equation}\label{eqybe}
  (R_{W,U} \otimes \id_V)  (\id_W \otimes R_{V,U})(R_{V,W} \otimes \id_U)= (\id_{U}\otimes R_{V,W})(R_{V,U}\otimes \id_W) (\id_V \otimes R_{W,U}).
 \end{equation}
 One can build such a map inductively, by starting with $R_{V_n,V_n}=R_q$ in \eqref{eq:Rmatrixdef}, defining $R_{V_n\td, V_n^{\otimes e}}$ by use of the formulas
\begin{equation}\label{eq:inductiveRmatrix}
\begin{split}
R_{X\otimes Y, Z} = ( R_{X,Z} \otimes \id_Y)(\id_X \otimes R_{Y,Z}), \\ 
R_{X, Y \otimes Z} = (\id_Y \otimes R_{X,Z}) ( R_{X,Y} \otimes \id_Z), 
\end{split}
\end{equation}
and then realizing any indecomposable degree $d$ representation of $\Uq$ as a subquotient of $V_n\td$. In the following, we denote $R_{V,V}$ by $R_V$. 

Similarly, given $V$ a polynomial $\Uq$-module of degree $d$ viewed as a representation of the coideal subalgebra $U_{Q,q}^B$, then there exists a $K$-matrix $K_V$ that is an $U_{Q,q}^B$-isomorphism and satisfies the reflection equation:
\begin{equation}\label{eqreflection}
(K_V \otimes \id_W)R_{W,V} (K_W \otimes \id_V)R_{V,W} = R_{W,V} (K_W \otimes \id_V)R_{V,W} (K_V \otimes \id_W). 
\end{equation}
Again, one can obtain the $K$-matrix on polynomial representations inductively, by starting with $K_{V_n}:=K_Q$ and using the formula:
\begin{equation}\label{eq:indK}
K_{V\otimes W} =  (K_V \otimes \id_W)R_{W,V} (K_W \otimes \id_V)R_{V,W}.
\end{equation}
In particular, this implies that $K_{V_n\td}$ is given by the action of $K_d K_{d-1}\cdots K_1$ on $V_n\td$, and for every subquotient $V$ of $V_n\td$, the $K$-matrix $K_V$ is obtained by restriction. 

In the Weyl group $W^A(de)$ with simple reflections $s_i, 1 \leq i \leq de-1$, consider the elements $w_i, 1 \leq i \leq d-1$ given 
in two line notation by
\begin{equation}\label{def:wi}   w_i = \bigl(\begin{smallmatrix}
    1 & \cdots & e(i-1) & ei-e+1 & \cdots & ei & ei+1& \cdots & ei+e-1 & ei+e & \cdots &de \\
    1 & \cdots & e(i-1) & ei+1 & \cdots & ei+e-1 & ei-e+1& \cdots & ei & ei+e & \cdots & de
  \end{smallmatrix}\bigr).
\end{equation}
Note that $w_i$ is the longest element in the parabolic subgroup (isomorphic to $W^A(e)$) in $W^A(de)$ generated by $s_{e(i-1)+1},\cdots,s_{ei-1}$.
 
Following \cite{BuciumasKo1}, we define $\mathcal{H}_{q}^A(d,e)$ as the subalgebra of $\mathcal{H}_{q}^A(de)$ generated by $T_{w_i}, 1 \leq i \leq d-1$. We call $\mathcal{H}_{q}^A(d,e)$ the $e$-Hecke algebra (of Coxeter type A). 

Let $V$ be a $\Uq$-module of degree $e$ and $R_V$ be its $R$-matrix. Then one can show (see the discussion after Definition 2.9 in \cite{BuciumasKo1}) that there is a right action of $\mathcal{H}_q^A(d;e)$ on $V^{\otimes d}$, where $T_{w_i}$ acts as $(R_{V})_{i,i+1}$. 

In the Weyl group $W^B(de)$ with simple reflections $s_i, 0 \leq i \leq de-1$, consider the elements $w_i \in W^A(de) \subset W^B(de), 1 \leq i \leq d-1$ defined in equation~\eqref{def:wi} and the element $w_0$ given by 
\begin{equation}\label{def:w0} 
w_0= s_0 (s_1 s_0 s_1) \cdots (s_{e-1} \cdots s_1 s_0 s_1 \cdots s_{e-1}). 
\end{equation}
Note that $w_0$ is the longest element in the parabolic subgroup (isomorphic to $W^B(e)$) in $W^B(de)$ generated by $s_0,\cdots,s_{e-1}$.

\begin{defn}
Define $\mathcal{H}^B_{Q,q}(d,e)$ as the subalgebra of $\mathcal{H}_{Q,q}^B(de)$ generated by $T_{w_i}, 0 \leq i \leq d-1$. We call $\mathcal{H}^B_{Q,q}(d,e)$ the two-parameter $e$-Hecke algebra of Coxeter type B. 
\end{defn}

\begin{rem}
The $e$-Hecke algebras are simple to define but not well understood. 
For example, the dimension of $\mathcal{H}_{Q,q}^B(1,2)$ is $4$ for $Q,q$ generic (and therefore larger than $\mathcal{H}_{1,1}^B(1,2)\cong kS_2$). This follows from the fact that the $K$-matrix $K_{V_n^{\otimes 2}} \in \End_{U_{Q,q}^B}(V_4^{\otimes 2})$ generates a subalgebra in $\End_{U_{Q,q}^B}(V_4^{\otimes 2})$ isomorphic to $\mathcal{H}_{Q,q}^B(1,2)$ (this is because the action of $\mathcal{H}_{Q,q}^B(1,2)$ on $(V_n^{\otimes 2})^{\otimes 1}$ is faithful for $n \geq 2$) and the $K$-matrix has $5$ different eigenvalues for $n \geq 4$.  Similarly, the dimension of $\mathcal{H}_{Q,q}^B(1,e)$ is equal to the number of different eigenvalues of $K_{V_{2e}^{\otimes e}}$. 
But computing the dimension of $\mathcal{H}_{Q,q}^{B}(d,e)$, for general $d$, seems like a hard problem. This is also the case for $e$-Hecke algebras of type A.
\end{rem}

Let $V$ be a $\Uq$-module of degree $e$ and let $K_V$ be its associated $K$-matrix. We call $V$ a type B $e$-Hecke $\emph{triple}$. The word triple comes from the fact that when we write $V$ we implicitly mean the triple $(
V,R_V,K_V)$, where we abbreviate $R_V=R_{V,V}$. 
\begin{lem}
There is a right action of $\mathcal{H}_{Q,q}^B(d,e)$ on $V\td$ where $T_{w_i}$ acts by $(R_{V})_{i,i+1}$ for $i>0$ and $T_{w_0}$ acts by $(K_V)_1$. 
\end{lem}
\begin{proof}
First we prove this for $V = V_n^{\otimes e}$. Then the elements $T_{w_i} \in \mathcal{H}_{Q,q}^{B}(d,e)$ act on $(V_n^{\otimes e})^{\otimes d} = V_n^{\otimes de}$ by $T_{w_i}= T_{s_{ie+e-1}} \cdots T_{s_{ie+1}} = (R_{V_n})_{ie+e-1,ie+e} \cdots (R_{V_n})_{ie+1,ie+2} = (R_{V_n^{\otimes e}})_{i,i+1}$ where the last equality involves the use of equation~\eqref{eq:inductiveRmatrix}.  A similar argument can be made for the $K$-matrix via equation~\eqref{eq:indK}. 

This means that $(K_{V_n^{\otimes e}})_1, (R_{V_n^{\otimes e}})_{i,i+1} \in \End (V_n^{\otimes e})^{\otimes d}$ satisfy all the relations the generators $T_{w_i}$ satisfy. 
A degree $e$ module of $\Uq$ is a subquotient of $V_n^{\otimes e}$ and therefore $(K_{V})_1, (R_{V})_{i,i+1} \in \End (V^{\otimes d})$ also satisfy the relations the generators $T_{w_i}$ satisfy, giving rise to an $e$-Hecke algebra representation. 
\end{proof}

Let us now turn our attention to defining generalized Schur algebras. We have already defined the Schur algebra of type $B$ in equation~\eqref{def:schuralgebraB}. Let $V, W$ be degree $e$ representations of $\Uq$. For every non-negative integer $d$ we define
\begin{equation}\label{defeSchuralgebra}
S_{Q,q}^B(V,W;d):= \Hom_{\mathcal{H}_{Q,q}^B(d,e)}(V^{\otimes d}, W^{\otimes d}). 
\end{equation}

In particular, we denote by $S_{Q,q}^B(n,m;d,e)$ the space $\Hom_{\mathcal{H}_{Q,q}^B(d,e)}((V_n^{\otimes e})^{\otimes d}, (V_m^{\otimes e})^{\otimes d})$ and let $S_{Q,q}^B(n;d,e)=S_{Q,q}^B(n,n;d,e)$. A relation between different Schur algebras and Hecke algebras is displayed in Figure~\ref{matingdragonflies}. The inclusions on the Hecke algebra side follow by definition, while the surjections on the Schur algebra side follow from the inclusions on the Hecke algebra side. 

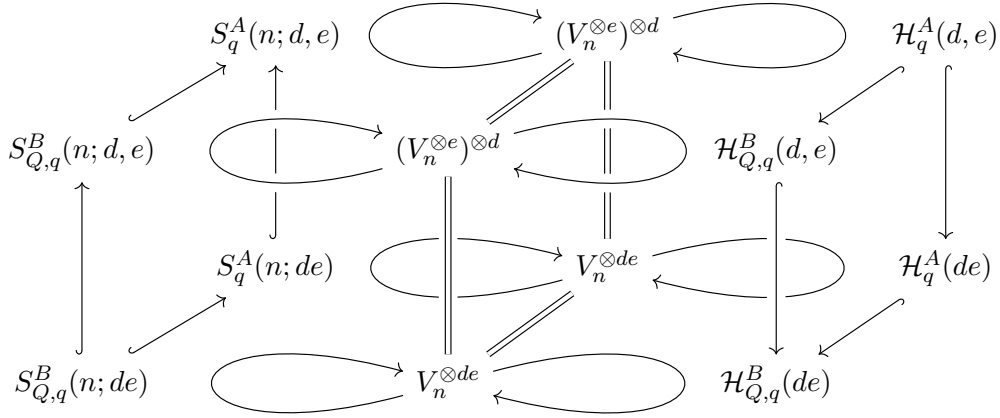
\begin{figure}
\begin{tikzcd}[row sep=2em, column sep=1em]
& S_q^A(n;d,e) \arrow[from=dl,hookrightarrow]  \arrow[from=dd,hookrightarrow]  &  &\arrow[loop left,distance=90]  (V_n^{\otimes e})^{\otimes d} \arrow[dd,equal] \arrow[loop right,distance=90] & &  \mathcal{H}_q^A(d,e) \arrow[dl,hook] \arrow[dd,hook] \\
S_{Q,q}^B(n;d,e) \arrow[from=dd,hookrightarrow] & &  \arrow[loop left,distance=90,crossing over] (V_n^{\otimes e})^{\otimes d} \arrow[loop right,distance=90,crossing over] \arrow[ur,equal,pos=0.35,crossing over]&  & \mathcal{H}_{Q,q}^B(d,e)   \\
& S_{q}^A(n;de) \arrow[from=dl, hookrightarrow]  &  & \arrow[loop left,distance=100]  V_n^{\otimes de} \arrow[loop right,,distance=100]& & \mathcal{H}_q^A(de) \arrow[dl,hook] \\
S_{Q,q}^B(n;de) & & \arrow[loop left,distance=100]   V_n^{\otimes de} \arrow[loop right,distance=100]  \arrow[uu,pos=0.5,equal,crossing over] \arrow[ur,equal,crossing over] & & \mathcal{H}_{Q,q}^B(de) \arrow[from=uu, crossing over,hook]\\
\end{tikzcd}\caption{On each row of the diagram above we have a commuting double action on the space $V_n^{\otimes de}$. A double centralizer property is satisfied for the double action on the bottom two rows for $Q,q$ generic. A question is whether the double action on the top two rows also satisfy a double centralizer property.}
\label{matingdragonflies}
\end{figure}

\section{Two-parameter quantum polynomial functors}\label{sec:quantumpolfunctors}

\subsection{Representations of categories}

Fix a field $k$. Let $\Lambda$ be a $k$-linear category.
A representation of $\Lambda$ is a $k$-linear functor $\Lambda \to \mathcal{V}$, where $\mathcal V$ is the category of finite dimensional $k$-vector spaces.

Let $\operatorname{mod}_{\Lambda}$ be the category of representations of $\Lambda$, where the morphism spaces are given by the natural transformations.

The following lemma and proposition are standard in homological algebra.

\begin{lem}\label{lemma:equivendring}
If $\Lambda$ consists of a single object $*$, then we have $\operatorname{mod}_\Lambda \cong \End_\Lambda(*)$-$\operatorname{mod} $.
\end{lem}
 Therefore we can think of $\operatorname{mod}_\Lambda$ as a generalization of the module category of an algebra.

\begin{defn}\label{generate}

A full subcategory $\Gamma$ of $\Lambda$ is said to {\it generate} $\Lambda$ if the additive Karoubi envelope of $\Gamma$ contains $\Lambda$. If $\Gamma$ consists of a single object $V$, we also say $V$ generates $\Lambda$.

\end{defn}

\begin{prop}\label{prop:representabilitygeneral}

If $\Gamma$ generates $\Lambda$, then the restriction functor $\operatorname{mod}_\Lambda\to \operatorname{mod}_\Gamma$ is an equivalence.

\end{prop}

For any inclusion of full subcategories $\Gamma\subseteq \Gamma'\subseteq\Lambda$, if $\Gamma$ generates $\Lambda$, then $\Gamma'$ generates $\Lambda$
As a consequence, the categories $\operatorname{mod}_\Gamma$, $\operatorname{mod}_{\Gamma'}$, $\operatorname{mod}_\Lambda$ are all equivalent.

In particular, if $V$ generates $\Lambda$, then $\operatorname{mod}_\Lambda$ is equivalent to $\End_\Lambda(V)$-$\operatorname{mod}$, the category of finite dimensional modules over the algebra $\End_\Lambda(V)$.

\begin{ex}\label{exFS}

The category of degree $d$ polynomial functors $\mathcal{P}^d$ can be defined as $\operatorname{mod}_{\Gamma^d\mathcal V}$ where $\Gamma^d\mathcal{V}$ is the category with objects vector spaces $V_n$ of dimension $n$ for any $n\geq 1$ and morphisms 
$\Hom_{\Gamma^d \mathcal{V}} (V_n,V_m):= \Hom_{S_d} (V_n^{\otimes d}, V_m^{\otimes d})$. If $n\geq d$, the object $V_n$ generates $\Gamma^d\mathcal V$. Note that the algebra $\End_{\Gamma^d\mathcal V}(V_n)=\End_{S_d}(V_n^{\otimes d})$ is the Schur algebra $S(n;d)$. It follows that $\mathcal P^d$ is equivalent to $\operatorname{mod} S(n;d)$ for all $n\geq d$. In this example we are dealing with the three categories $\Lambda=S_d$-$\operatorname{mod}\supset \Gamma' =\Gamma^d\mathcal V\supset \Gamma = \{V_n\}$, viewing $\Gamma^d\mathcal{V}$ as a full subcategory of $S_d$-mod consisting of the objects of the form $V_n\td$ for all $n$. 
\end{ex}

In fact, all variations of the category of polynomial functors, including what we present in this work, can be identified with module categories of some interesting algebras by use of Lemma~\ref{lemma:equivendring} and Proposition \ref{prop:representabilitygeneral}. Example \ref{exFS} is a classical result of Friedlander and Suslin~\cite{FS}. The next example is the quantum polynomial functors of Hong and Yacobi \cite{HY}, which provide a quantization of Example \ref{exFS}.

\begin{ex}\label{ex:qpf}
Let us denote by $\mathcal{AP}_q^d$ the category defined as $\operatorname{mod}_{\Gamma_q^d\mathcal V}$, where $\Gamma^d_q\mathcal{V}$ is the category with objects vector spaces $V_n$ of dimension $n$ for any $n\geq 1$ and morphisms 
$\Hom_{\Gamma^d_q \mathcal{V}} (V_n,V_m):= \Hom_{\mathcal{H}^A_q(d)} (V_n^{\otimes d}, V_m^{\otimes d})$ where $\mathcal{H}^A_q(d)$ acts on $V_n^{\otimes d}$ via $R$-matrices as in equation~\eqref{eq:Rmatrixdef}. 
As in the non-quantum case, we have that $\End_{\Gamma_q^d\mathcal V}(V_n)=\End_{\mathcal{H}_q^A(d)}(V_n^{\otimes d})=S^A_q(n;d)$ and $\mathcal{AP}_q^d$ is equivalent to $\operatorname{mod} S_q^A(n;d)$ for all $n\geq d$. We rename $\Gamma^d\mathcal V$ to $\mathcal C^A_d$.
\end{ex}

\subsection{Polynomial functors and type B Hecke algebras}\label{sec:defofpolfunct}

\begin{defn}
The quantum divided power category $\mathcal{C}^B_d$ has objects $V_n$ for $n \geq 1$. The morphisms in this category are 
\[ \Hom_{\mathcal{C}_d^B} (V_n, V_m) := \Hom_{\mathcal{H}_{Q,q}^B(d)}(V_n^{\otimes d}, V_m^{\otimes d}).\] 
\end{defn}
Equivalently, we can define $\mathcal C^B_d$ as the full subcategory of $\mathcal{H}_{Q,q}^B(d)$-mod consisting of the objects $V_n\td$ for all $n$.
\begin{defn}\label{defntypeBpolfct}
We define the category of  type BC polynomial functors as 
\[\mathcal{P}_{Q,q}^d:=\mod_{\mathcal{C}^B_d}.\]
\end{defn}

Note that by definition, every $F\in \mathcal{P}_{Q,q}^d$ induces a linear map 
\[F: \Hom_{\mathcal{C}_d^B}(V_n,V_m) \to \Hom_{k}(F(V_n),F(V_m)).
\]

\begin{prop}
Let $F \in \mathcal{P}_{Q,q}^d$. The space $F(V_n)$ has the structure of a $S_{Q,q}^B(n)$-module.  
\end{prop}
\begin{proof}
Given an element $x \in S_{Q,q}^B(n;d) = \Hom_{\mathcal{H}_{Q,q}^B(d)}(V_n^{\otimes d}, V_n^{\otimes d})$, there is a corresponding element $F(x) \in \End(F(V_n))$. Since the functor $F$ is linear, the space $F(V_n)$ has the structure of an $S_{Q,q}^B(n;d)$-module with $x \in S_{Q,q}^B(n;d)$ acting on $F(V_n)$ via $F(x)$.
\end{proof}

From Remark~\ref{coidealschur}, the Schur algebra $S_{Q,q}^B(n;d)$ is a quotient of the coideal $U_{Q,q}^B$ in the generic case. It follows that $F(V_n)$ is endowed with the structure of a $U_{Q,q}^B$-module of degree $d$.

\subsection{Representability}\label{sec:representability}

We now show that the category $\mathcal{P}_{Q,q}^d$ is equivalent, under certain conditions, to the module category over the finite dimensional algebra $S_{Q,q}^B(n;d) = \End_{\mathcal{H}^B_{Q,q}(d)}(V_n^{\otimes d})$. 
This follows from Lemma~\ref{lemma:equivendring} and Proposition \ref{prop:representabilitygeneral} if we prove that the domain category $\mathcal{C}^B_d$ is generated by the object $V_{n}$ in the sense of Definition \ref{generate}.

We split this section into two parts depending on the parity of $n$. In \S\ref{sec:representabilityodd} we show the equivalence between $\mathcal{P}_{Q,q}^d$ and $S_{Q,q}^B(n;d)$-$\operatorname{mod}$ for $n$ \textit{odd}. In \S\ref{sec:representabilityeven} we impose the condition that $Q,q$ are generic and prove the equivalence for all $n$. We explain in Remark~\ref{remgenerationeven} what can go wrong if $n$ is even. 

As a convenient convention for the proof, we say for two objects $V,W\in \Lambda=\mathcal{C}^B_d$ that {\it $V$ generates $W$} if $W$ is a direct summand of a direct sum of $V$. We say that $V$ generates $\Lambda$ if $V$ generates every object in $\Lambda$. This definition is consistent with Definition~\ref{generate}.

\subsubsection{Representability for $n$ odd}\label{sec:representabilityodd}

Let $r$ be a non-negative integer. 
\begin{prop}\label{prop:Vdgeneration}
The object $V_n$ generates $\mathcal{C}_d^B$ if $n=2r+1 \geq 2d$. 
\end{prop}
\begin{proof}
Let $n=2r+1 \geq 2d$. We want to show that $V_n$ generates $V_m$ for all $m$. Note that $V_n\td$ is a direct sum of $\Hb$-modules $V(\mathbf{a},n)$ and $V_m\td$ is a direct sum of modules $V(\mathbf{b},m)$. By Lemma \ref{lemma:generationodd}, for every $V(\mathbf{b},m)$ there is a $V(\mathbf{a},n)$ such that the two spaces are isomorphic as $\Hb$-modules. It follows by definition that $V_n$ generates $V_m$ for all $m$ which implies that $V_n$ generates $\mathcal{C}_d^B$.
\end{proof}

The following result relates the category of two-parameter polynomial functors with the category of modules of the type B Schur algebra. 

\begin{thm}\label{thm:representability}
The category $\mathcal{P}_{Q,q}^d$ is equivalent to the category of finite dimensional modules of the endomorphism algebra $S_{Q,q}^B(n;d)$ where $n=2r+1$ for any $r \geq d$. 
\end{thm}
\begin{proof}
Use Proposition \ref{prop:Vdgeneration} to apply Proposition \ref{prop:representabilitygeneral} and Lemma~\ref{lemma:equivendring} with $\Gamma=\{V_{2r+1}\}$ and recall that $S_{Q,q}^{B}(2r+1;d) = \End_{\Gamma} (V_{2r+1})$.
\end{proof}

\begin{cor}\label{cor:moritaeqgeneral}
The Schur algebras $S_{Q,q}^B(m;d)$ and $S_{Q,q}^B(n;d)$ are Morita equivalent if $m,n \geq 2d$ are odd. 
\end{cor}

\subsubsection{Representability for $n$ even}\label{sec:representabilityeven}

We now assume $Q,q$ are generic, which implies the Hecke algebra $\HB$ is semisimple.

\begin{lem}\label{lemma:generationeven}
Suppose $\HB$ is semisimple. Then $V_{2m}$ generates $V_{2m-1}$. 
\end{lem}
\begin{proof}
It is enough to find a summand in $V_{2m}\td$ which is isomorphic to $V(\mathbf a)=V(\mathbf a, 2m-1 )$ for an arbitrary $\mathbf a\in \mathbb I^d_{2m-1}$. In fact, since $\mathcal{H}_{Q,q}^B(d)$-modules are completely reducible, it is enough to construct an injective map from $V(\mathbf a)$ into $V_{2m}\td$. 
Since $V(\mathbf a)=V(w \mathbf a)$ for $w\in W^B(d)$, we may assume that $0\leq a_1\leq\cdots\leq a_d$. Let $a_{i+1}$ be the first entry greater than zero.

Let $a'_j=a_j+\frac{1}{2}$.
We define 
\begin{equation}\label{eq:l0l1}
\begin{split}
    &\ell_0(w)= \text{the multiplicity of $s_0$ in a reduced expression of }w;\\
    &\ell_1(w)=\ell(w)-\ell_0(w),
\end{split}
\end{equation}
 where $\ell(w)$ is the Coxeter length for $W^B(d)$.
Then define the element 
\[\bar v_{\mathbf{a}}:=\sum_{w\in W^{B}(i)/\operatorname{Stab}_{W^B(i)}(\frac{1}{2},\cdots,\frac{1}{2})} Q^{-\ell_0(w)}q^{-\ell_1(w)}v_{w(\frac{1}{2},\cdots,\frac{1}{2})}\otimes v_{a'_{i+1}}\otimes\cdots\otimes v_{a'_{d}}\] 
in $V_{2m}\td$. Here $v_{(\frac{1}{2},\cdots,\frac{1}{2} )} := v_{\frac{1}{2}} \otimes \cdots \otimes v_{\frac{1}{2}}$, where there are $i$ terms in the tensor product and in $(\frac{1}{2},\cdots,\frac{1}{2})$. The group $W^B(d)$ acts as in equation~\eqref{WeylgroupactionBC}. 

The vector $\bar v_{\mathbf{a}}$ is an eigenvector with eigenvalue $q^{-1}$ for $T_j \in \mathcal{H}_{Q,q}^B(d), 0<j\leq i$ and eigenvalue $Q^{-1}$ for $T_0$, just like $v_\mathbf{a} = v_{(0,\cdots,0)} \otimes v_{a_{i+1}} \otimes \cdots \otimes v_{a_d}$.
Therefore the element $\bar v_{\mathbf{a}}$ has the same stabilizer in $\mathcal{H}_{Q,q}^B(d)$ as $v_\mathbf{a}$ and the assignment $v_{\mathbf{a}}\mapsto \bar v_{\mathbf{a}}$ induces a well-defined $\mathcal{H}^B_{Q,q}(d)$-map $V(\mathbf a)\to V_{2m}\td$ which is injective.
\end{proof}

\begin{lem}\label{2dgenerates2d+1}
Suppose $\HB$ is semisimple. Then $V_{2d}$ generates $V_{2d+1}$. 
\end{lem}
\begin{proof}
The proof uses the same arguments as in the proof of Lemma~\ref{lemma:generationeven}. 
We note it does not hold in general that $V_{2m}$ generates $V_{2m+1}$.   
\end{proof}

\begin{thm}\label{thm:representabilityss}
Let $Q,q$ be generic. The category $\mathcal{P}_{Q,q}^d$ is equivalent to the category of finite dimensional modules of the endomorphism algebra $S_{Q,q}^B(n;d)$ where $n \geq 2d$. 
\end{thm}

\begin{proof}
The Hecke algebra $\HB$ is semisimple because we work with $Q,q$ be generic. The case when $n$ is odd has been proved in greater generality, so we focus on $n=2m+2$. Using Lemma~\ref{lemma:generationeven}, $V_{2m+2}$ generates $V_{2m+1}$, which by Proposition~\ref{prop:Vdgeneration} and transitivity implies that $V_{2m+2}$ generates $\mathcal{C}_d^B$. This argument proves the statement for $n \geq 2d+1$ and Lemma~\ref{2dgenerates2d+1} improves the bound to $n\geq2d$. The rest of the proof is the same as for Theorem~\ref{thm:representability}.   
\end{proof}

\begin{cor}\label{cor:moritaeq}
Let $Q,q$ be generic. The Schur algebras $S_{Q,q}^B(m;d)$ and $S_{Q,q}^B(n;d)$ are Morita equivalent if $m,n \geq 2d$. 
\end{cor}

\begin{rem}\label{remgenerationeven}
When $Q$ or $q$ is a root of unity (or when char$(k)=2$) Lemma \ref{lemma:generationeven} fails. To exemplify this, take $Q^2=-1$ and $d=1$ in Lemma \ref{lemma:generationeven}. Then $V_1$ is an $\mathcal{H}_{Q,q}^B(1)$-submodule of $V_2$, but it is not a quotient. This is because $K_{Q}: V_2 \to V_2$ is not diagonalizable when $Q^2=-1$. When $q^2=-1$, similar phenomena happen with $R_q$  for $d \geq 2$.
\end{rem}

\subsection{Stability for quantum symmetric pairs and Schur algebras}\label{stability}
Corollary \ref{cor:moritaeq} allows us to state a stability property for the Schur algebra $S_{Q,q}^B(n;d)$ as $n\to \infty$.
This extends to a property of the coideal subalgebra $U^B_{Q,q}$. 

Let us consider $U_{Q,q}^B$ in the $n=2r$ case. The degree $d$ irreducibles of $U_{Q,q}^B(\mathfrak{gl}(2r))$ are indexed by pairs of partitions $(\lambda, \mu)$ such that $|\lambda| + |\mu|=d, l(\lambda) \leq r, l(\lambda) \leq r$. There is a notion of compatibility for degree $d$ polynomial representations of $U_{Q,q}^B(\mathfrak{gl}(2r))$ for different $r$, which allows us to take the limit $r \to \infty$. Corollary \ref{cor:moritaeq} implies that the limit of the polynomial representation theory of degree $d$ as $r\to \infty$ is well defined and that it is equivalent to the representation theory of $S^B_{Q,q}(n;d)$ for any $n\geq 2d$.

Let us be more precise. 
Let $\mathbb{I}_{2\infty} = \mathbb{Z} + \frac{1}{2}$ and let $\mathbb{I}_{2\infty+1} = \mathbb{Z}$ and $V_{2\infty}$ and $V_{2\infty +1}$ be vector spaces with basis indexed by elements in $\mathbb{I}_{2\infty}$ and $\mathbb{I}_{2\infty+1}$, respectively. Define the quantum groups $U_q(\mathfrak{gl}(2\infty))$ and $U_q(\mathfrak{gl}(2\infty+1))$ via generators and relations as in equation \eqref{qgrelations} with $V_{2\infty}$ and $V_{2\infty+1}$ as defining representations, respectively (see for example \cite[Section 7]{ES}). Then we define the coideal subalgebras $U^B_{Q,q}(2\infty), U^B_{Q,q}(2\infty+1)$ by extending the definition in the finite case to the infinite case. There is an obvious extension of the right action of $\mathcal{H}_{Q,q}^B(d)$ on $V_n\td$ in equation \eqref{HeckeactionBC} to when $n$ gets replaced by $2\infty$ or $2\infty+1$, therefore allowing us to define the following Schur algebras:
\begin{equation}
\begin{split}
S_{Q,q}^B(2\infty;d) &:= \End_{\mathcal{H}_{Q,q}^B(d)}(V_{2\infty}^{\otimes d}),\\
S_{Q,q}^B(2\infty+1;d) &:= \End_{\mathcal{H}_{Q,q}^B(d)}(V_{2\infty+1}^{\otimes d}).
\end{split}
\end{equation}

\begin{rem}
The coideal subalgebras $U^B_{Q,q}(2\infty), U^B_{Q,q}(2\infty+1)$ have specialization $Q \to 1$ and $Q\to q$ as in the finite case. These infinite versions are compatible with combinatorics of translation functors and can be categorified in a way that they have categorical actions on representation categories of type BD (see \cite[Section 7]{ES}). 
\end{rem}

We define the $\emph{polynomial}$ representations of $S_{Q,q}^B(2\infty;d)$ and $S_{Q,q}^B(2\infty+1;d)$ as the representations appearing as subquotients of the representations $V_{2\infty}\td$ and $V_{2\infty+1}\td$, respectively.
We can show via essentially the same technique as above that Theorem~\ref{thm:representability} and Corollary \ref{cor:moritaeq} extend to the $2\infty / 2\infty +1$ case:

\begin{prop}\label{prop:moritaequivalenceinfinity}
The category of polynomial representations of the Schur algebras $S^B_{Q,q}(2\infty;d)$ and that of $S^B_{Q,q}(2 \infty +1;d)$ are both  equivalent to the category  $\mathcal{P}_{Q,q}^d$. 
\end{prop}

Define the $\emph{polynomial representation theory}$ of $U^B_{Q,q}(2 \infty)$ and $U^B_{Q,q}(2 \infty+1)$ as a direct sum of the categories 
\begin{equation}
\begin{split}
\mathcal{P}_{Q,q}(2\infty):=&\bigoplus_{d\geq1}\mathcal{P}_{Q,q}^d(2 \infty) = \bigoplus_{d\geq1}S_{Q,q}^{B}(2 \infty;d)\operatorname{-mod}, \\
\mathcal{P}_{Q,q}(2\infty+1):=& \bigoplus_{d\geq1}\mathcal{P}_{Q,q}^d(2 \infty+1) = \bigoplus_{d\geq1}S_{Q,q}^{B}(2 \infty+1;d)\operatorname{-mod}.
\end{split}
\end{equation}
The following theorem follows immediately from Proposition \ref{prop:moritaequivalenceinfinity}.

\begin{thm}\label{thmlimitoddeven}
The categories $\mathcal{P}_{Q,q}(2\infty)$ and $\mathcal{P}_{Q,q}(2\infty+1)$ are equivalent. 
\end{thm}

The theorem implies that the polynomial representation theory of the coideal subalgebras in the $n \to \infty$ limit does not depend on the parity of $n$. Therefore one can replace $\mathcal{P}_{Q,q}(2\infty)$ and $\mathcal{P}_{Q,q}(2\infty+1)$ by $\mathcal{P}_{Q,q}(\infty)$.

\begin{rem}\label{rem:noddevencoideal}
Note that there is a difference between the definition of $U_{Q,q}^B(\mathfrak{gl}(n))$ for odd and for even $n$. On the level of generators \eqref{eq:coidealgen}, when $n-1$ is odd, the coideal has a special generator $t$, while when $n-1$ is even, the generators $e_{\half}, f_\half$ are special. 
When $n=2r$, the coideal subalgebra $U_{Q,q}^B \subset \Uq$ is a quantization of the subalgebra $U(\mathfrak{gl}(r) ) \oplus U(\mathfrak{gl}(r)) \subset U(\mathfrak{gl}(2r))$. When $n=2r+1$, the coideal subalgebra $U_{Q,q}^B \subset \Uq$ is a quantization of the subalgebra $U(\mathfrak{gl}(r) ) \oplus U(\mathfrak{gl}(r+1)) \subset U(\mathfrak{gl}(2r+1))$. This difference persists even in the $n=2\infty$ vs $n=2\infty +1$ case. Therefore it is unclear how to relate the coideals $U^B_{Q,q}(2 \infty)$ and $U^B_{Q,q}(2 \infty+1)$ as algebras. 
\end{rem}

\section{Polynomial functors and braided categories with a cylinder twist}\label{sec:cylindertwist}

\subsection{Actions of monoidal categories}\label{sec:action}

Let $\mathcal{B}$ be a category and let $(\mathcal{A},\otimes,1_\mathcal{A})$ be a monoidal category. Denote by $l_X:1_{\mathcal{A}}\otimes X \to X$ the left unitor. Denote by $a_{X_1,X_2,X_3}:(X_1 \otimes X_2) \otimes X_3\to X_1\otimes (X_2 \otimes X_3)$ the associativity morphism of $\mathcal{A}$.

\begin{defn}\label{def:action}
We say $\mathcal{A}$ acts on $\mathcal{B}$ (from the right) if there is a functor $*:\mathcal{B} \times \mathcal{A} \to \mathcal{B}$ such that 
\begin{enumerate}
\item for morphisms $f_1, f_2$ in $\mathcal{B}$ and morphisms $g_1, g_2$ in $\mathcal{A}$ the equation
\[(f_1 * g_1)(f_2 * g_2) = (f_1 f_2)*(g_1 g_2) \]
holds whenever both sides are defined.
\item There is a natural morphism $\lambda:*(\operatorname{id} \times \otimes) \to *(* \times \operatorname{id})$, i.e., $\lambda_{Y,X_1,X_2}:Y * (X_1\otimes X_2) \to (Y * X_1) * X_2$ such that the following diagram commutes: 

\[\begin{tikzpicture}
  \matrix (m) [matrix of math nodes,row sep=3em,column sep=8em,minimum width=2em]
  {
      Y*((X_1 \otimes X_2) \otimes X_3) & Y*(X_1 \otimes (X_2 \otimes X_3)) \\
      & (Y * X_1) *(X_2 \otimes X_3) \\
     (Y*(X_1 \otimes X_2)) *X_3 & ((Y*X_1)*X_2)*X_3\\};
  \path[-stealth]
    (m-1-1) edge node [left] {$\lambda_{Y,X_1 \otimes X_2,X_3}$} (m-3-1)
    (m-1-1) edge node [above] {$\operatorname{id}_Y*a_{X_1,X_2,X_3}$} (m-1-2)
    (m-1-2) edge node [right] {$\lambda_{Y,X_1, X_2\otimes X_3}$} (m-2-2)
    (m-2-2) edge node [right] {$\lambda_{Y*X_1,X_2,X_3}$} (m-3-2)
    (m-3-1) edge node [below] {$\lambda_{Y,X_1, X_2}*\operatorname{id}_{X_3}$} (m-3-2);
\end{tikzpicture}\]
\item There is a natural isomorphism $\rho_{Y}:Y*1_{\mathcal{A}} \to Y$ such that the following diagram commutes:
\[\begin{tikzpicture}
  \matrix (m) [matrix of math nodes,row sep=3em,column sep=6em,minimum width=2em]
  {
      Y*(1_{\mathcal{A}} \otimes X) & (Y*1_{\mathcal{A}}) *X \\
      Y*X & Y * X\\};
  \path[-stealth]
    (m-1-1) edge node [above] {$\lambda_{Y,1,X}$} (m-1-2)
    (m-1-1) edge node [left] {$\operatorname{id}_{Y}*l_X$} (m-2-1)
    (m-1-2) edge node [right] {$\rho_Y*\operatorname{id}_X$} (m-2-2)
    (m-2-1) edge node [below] {$\operatorname{id}_{Y*X}$} (m-2-2);
\end{tikzpicture}
\]
\end{enumerate} 
\end{defn}

Following \cite{HO}, we call the triple $(\mathcal{B}, \mathcal{A},*)$ an $\emph{action pair}$. We write $(\mathcal{B}, \mathcal{A})$ for $(\mathcal{B}, \mathcal{A},*)$ if it is clear what the action $*$ is.   
 
Consider the category of type A quantum polynomial functors $\mathcal{AP}_q = \bigoplus_d \mathcal{AP}_q^d$ defined in Example~\ref{ex:qpf}. The category $\mathcal{AP}_q$ has a monoidal structure. Given $F \in \mathcal{AP}^d_q$ and $G \in \mathcal{AP}^e_q$, define $F \otimes G \in \mathcal{AP}^{d+e}_q$ as
$ F \otimes G (V_n):= F(V_n) \otimes G(V_n)$ and on the morphisms, $F \otimes G$ is given as the composition 
\begin{equation}\label{compositionA}
\begin{split}
& \Hom_{\mathcal{H}_q^A(d+e)}(V_n^{\otimes d+e}, V_m^{\otimes d+e}) \to \Hom_{\mathcal{H}_q^A(d)\otimes \mathcal{H}_q^A(e)}(V_n^{\otimes d}\otimes V_n^{\otimes e}, V_m^{\otimes d}\otimes V_m^{\otimes e}) \\
&\to \Hom_{\mathcal{H}_q^A(d)}(V_n^{\otimes d},V_m^{\otimes d}) \otimes \Hom_{\mathcal{H}_q^A(e)}(V_n^{\otimes e}, V_m^{\otimes e}) \\
&\to \Hom(F(V_n),F(V_m))\otimes \Hom(G(V_n),G(V_m)) \\ 
& \to \Hom(F\otimes G(V_n), F \otimes G(V_m)).  
\end{split}
\end{equation}

There is also a unit with respect to this monoidal structure. The unit is a degree $0$ polynomial functor, which we denote by $1_{\mathcal{AP}_q}$ and is defined $1_{\mathcal{AP}_q} (V_n):=k$ and on morphisms it maps $f \in \Hom_{\mathcal{H}_q^A(d)}(V_n^{\otimes 0},V_m^{\otimes 0}) \simeq \Hom (k,k)$ identically to $\Hom(k,k)$.  

Given $F \in \mathcal{AP}_q^d, G \in \mathcal{AP}_q^e$, the functoriality of $F, G$ endows the spaces $F(V_n)$ and $G(V_n)$ with actions of the $q$-Schur algebras $S_q(n;d)$ and $S_q(n;e)$, respectively, or equivalently, degree $d$ (respectively, degree $e$) $\Uq$-module structures. 

The category $\mathcal{AP}_q$ is a braided monoidal category with the braiding:
\begin{equation}\label{eqbraiding}
R_{F,G}:F \otimes G \to G \otimes F, 
\end{equation}
where $R_{F,G}(V_n) := R_{F(V_n), G(V_n)}$ is the $R$-matrix defined in \S~\ref{sec:RKe}. 
This is proved in \cite[Theorem 5.2]{HY}.

Recall the category $\mathcal{P}_{Q,q}$ defined in Definition~\ref{defntypeBpolfct}.  
\begin{thm}\label{thm:strictactionpair}
The pair $(\mathcal{P}_{Q,q}, \mathcal{AP}_q)$ is an action pair.
\end{thm}
\begin{proof}
Let us first define the action of $\mathcal{AP}_q$ on $\mathcal{P}_{Q,q}$. Let $F \in \mathcal{AP}^d_q$ and $G \in \mathcal{P}^e_{Q,q}$. Define $G*F\in \mathcal{P}^{d+e}_{Q,q}$ on objects as $G*F(V_n) := G(V_n) \otimes F(V_n)$ and on morphisms as the composition:
\begin{equation}\label{compositionB}
\begin{split}
& \Hom_{\mathcal{H}_q^B(d+e)}(V_n^{\otimes d+e}, V_m^{\otimes d+e}) \to \Hom_{\mathcal{H}_q^B(d)\otimes \mathcal{H}_q^A(e)}(V_n^{\otimes d}\otimes V_n^{\otimes e}, V_m^{\otimes d}\otimes V_m^{\otimes e}) \\
&\to \Hom_{\mathcal{H}_q^B(d)}(V_n^{\otimes d},V_m^{\otimes d}) \otimes \Hom_{\mathcal{H}_q^A(e)}(V_n^{\otimes e}, V_m^{\otimes e}) \\
& \to \Hom(G(V_n),G(V_m))\otimes \Hom(F(V_n),F(V_m)) \\ 
& \to \Hom(G* F(V_n), G*F(V_m)).  
\end{split}
\end{equation}

Since we have defined $G*F(V_n) := G(V_n) \otimes F(V_n)$, the natural morphisms $\lambda_{Y,X_1,X_2}:Y * (X_1\otimes X_2) \to (Y * X_1) * X_2$ and $\rho_{Y}:Y*1_{\mathcal{A}} \to Y$ are the identity maps on objects.

Using the action defined above, the proof consists only of routine verification of the axioms.

For example, let us prove the first property in Definition~\ref{def:action}. Given $f:F_1 \to F_2$ and $g:G_1 \to G_2$, denote by $f_{V_n}: F_1(V_n) \to F_2(V_n)$ and $g_{V_n}: G_1(V_n) \to G_2(V_n)$ their values on objects, respectively. Then $f * g : F_1 * G_1 \to F_2 * G_2$ is given on objects by $f*g_{V_n} = f_{V_n} \otimes g_{V_n}$. The first property then becomes equivalent to the equation $((f_1)_{V_n} \otimes (g_1)_{V_n}) ((f_2)_{V_n} \otimes (g_2)_{V_n}) = ( (f_1)_{V_n} (f_2)_{V_n} \otimes (g_1)_{V_n} (g_2)_{V_n} )$ which is a standard property of tensor product. 

We omit the rest of the proofs since they are routine.
\end{proof}

\begin{rem}
The action in Theorem~\ref{thm:strictactionpair} is a right action. This fact is related to the coideal $U_{Q,q}^B$ being a right coideal, i.e. $\Delta(U_{Q,q}^B) \subset U_{Q,q}^B \otimes \Uq$ and to the fact that $T_0\in \HB$ acts on the first (left) component of $V_n\td$. There is a version of the Schur-Weyl duality in Theorem ~\ref{thm:doublecentralizerB} where the Hecke algebra generator $T_0$ acts on the last component of $V_n\td$ (and $T_1$ acts on the last two components of $V_n\td$ etc.) and the corresponding coideal is a $\emph{left}$ coideal. The action pair in Theorem~\ref{thm:strictactionpair} is defined similarly, but it is now a $\emph{left}$ action pair. 
\end{rem}

\begin{rem}
The action in Theorem~\ref{thm:strictactionpair} is bilinear. We can therefore say that $\mathcal P_{Q,q}$ is a (right) module for $\mathcal{AP}_q$.  
\end{rem}

\subsection{Cylinder braided action pairs}\label{sec:cylbraiddd}

In this subsection we show how to build a cylinder braided action pair from the theory of two-parameter quantum polynomial functors. 

\begin{defn}\label{defcylbraidedacpair}
An action pair $(\mathcal{B}, \mathcal{A})$ is said to be $\emph{cylinder braided}$ if:
\begin{enumerate}
\item There exists an object $1\in\mathcal B$ which gives a bijection $\operatorname{Ob}(\mathcal{A}) \to \operatorname{Ob}(\mathcal{B})$ via $X\mapsto 1*X$. 
\item $\mathcal{A}$ is a braided monoidal category with braiding $c$.
\item There exists a natural isomorphism $t :\id_\mathcal{B}\to \id_\mathcal{B}$ such that the following equalities hold:
\[ c_{Y,X}(t_Y \otimes \id_X)c_{X,Y}(t_X \otimes \id_Y) = (t_X \otimes \id_Y)c_{Y,X} (t_Y \otimes \id_X)c_{X,Y} = t_{X \otimes Y}.\]
\end{enumerate}
\end{defn}

The goal of this subsection is to show that the $\mathcal{AP}_q$ action on $\mathcal P_{Q,q}$ produces a cylinder braided action pair.
The module category $\mathcal B$ here consists of the (one-parameter) quantum polynomial functors viewed as two-parameter quantum polynomial functors. We make this more precise:

Recall that $\mathcal{AP}^d_q=\mod_{\mathcal C^A_d}$ and $\mathcal{P}^d_{Q,q}=\mod_{\mathcal C^B_d}$, and that $Ob (\mathcal{C}^B_d) = Ob (\mathcal{C}^A_d)$. The Hecke algebra inclusion $\mathcal{H}_q^A(d) \xhookrightarrow{} \Hb$ implies the inclusion $\Hom_{\mathcal{H}_{Q,q}^B(d)}(V_n\td, V_m\td)  \xhookrightarrow{} \Hom_{\mathcal{H}_q^A(d)}(V_n\td, V_m\td)$ which is the same as the inclusion $\operatorname{Mor}_{\mathcal{C}_d^B}(V_n, V_m) \xhookrightarrow{} \operatorname{Mor}_{\mathcal{C}_d^A}(V_n, V_m)$.
We thus have the restriction functor  \[\res: \mathcal{AP}_q \to \mathcal{P}_{Q,q}.\]
The functor $\res$ is equivalent to the restriction of $S_q^A(n;d)$-modules to $\SB(n;d)$-modules in view of Theorem \ref{thm:representability}. 

Denote by $\res(\mathcal{AP}_q)$ the full subcategory of $\mathcal P_{Q,q}$ whose objects are $\res Ob(\mathcal P_{Q,q})$. We define an action of $\mathcal{AP}_q$ on $\res(\mathcal{AP}_{q})$ similar to the action defined in \S~\ref{sec:action}.
Let $F\in \res(\mathcal{AP}_{q})$ and $G \in \mathcal{AP}^e_{q}$. There is a unique $F' \in \mathcal{AP}_q^d$ such that $F=\res(F')$. Define $F*G\in \res \mathcal{AP}^{d+e}_{q}$ as $\res(F*G)'$, where  $(F*G)'\in \mathcal{AP}_q^{d+e}$ is $(F*G)' := F' \otimes G$. 


Recall the element $c_K=c_K^d=\prod_i K_i\in \HB$. Lemma \ref{cKcentral} implies $c_K \in \Hom_{\HB}(V_n\td, V_m\td)$.

Given an element $F \in \res\mathcal{AP}_{q}$, define $K_F:F\to F$ by
\[K_F(V_n):=  F(c_K):F(V_n) \to F(V_n).\]

\begin{lem}\label{lemnaturaltransfK}
The map $K_F$ is a morphism in the category $\res\mathcal{AP}_{q}$. 
\end{lem}
\begin{proof}
Assume $F$ is of degree $d$. To see that $K_F$ is a morphism, we need to show that the following diagram commutes
\[\begin{tikzpicture}
  \matrix (m) [matrix of math nodes,row sep=2em,column sep=6em,minimum width=2em]
  {
      F(V_n) & F(V_n) \\
      F(V_m) & F(V_m)\\};
  \path[-stealth]
    (m-1-1) edge node [above] {$F(c_K)$} (m-1-2)
    (m-1-1) edge node [left] {$F(x)$} (m-2-1)
    (m-1-2) edge node [right] {$F(x)$} (m-2-2)
    (m-2-1) edge node [below] {$F(c_K)$} (m-2-2);
\end{tikzpicture}
\]
for all $x \in \Hom_{\HB}(V_n\td, V_m\td)$. Since $c_K\in \HB$, it commutes with $x$. Thus we have $F(x)F(c_K) = F(x c_K) =F(c_K x) = F(c_K)F(x)$. The statement of the lemma follows. 
\end{proof}

\begin{thm}\label{thmactioncylbraided}
The action pair $(\res(\mathcal{AP}_{q}), \mathcal{AP}_q)$ is a cylinder braided action pair.
\end{thm}
\begin{proof}
The action in Theorem~\ref{thm:strictactionpair} preserves $\res(\mathcal{AP}_q)$. Thus $( \res(\mathcal{AP}_{q}), \mathcal{AP}_q)$ is an action pair by restriction.

To show that the action pair is cylinder braided, we let $1:=\res k\in \res(\mathcal{AP}_q)$, where $k\in\mathcal{AP}_q$ is the tensor identity (the constant functor) and identify $F\in Ob(\mathcal{AP}_q)$ with $\res F \in Ob(\res \mathcal{AP}_q)$. Take $c_{F,G}$ to be the braiding of $\mathcal{AP}_q$ in \eqref{eqbraiding} and set $t_F = K_F$. 
To prove that $t$ is a natural transformation, let $f \in \operatorname{Mor}_{\mathcal{AP}_q}(F,G)$. This means that
\[ f_{V_n} F(x) = G(x) f_{V_n}\]
for any $x \in \operatorname{Mor}_{\mathcal{C}_d^A}(V_n,V_n)$. Since $K_F(V_n)=F(c_K)$, taking $x=c_K$ gives what we need. 

To show the relation \[ R_{G,F} (K_G \otimes \id_F)R_{F,G} (K_F \otimes \id_G)=
(K_F \otimes \id_G)R_{G,F} (K_G \otimes \id_F)R_{F,G} =
K_{F \otimes G},\]
it is enough to consider the case $F=\otimes^d$ and $G=\otimes^e$ since the morphisms $R, K$ restrict to subobjects. 
Since $R_{\otimes^d,\otimes^e}$ is given by the action of $T_{d,e}$, the above relation is equivalent to the equation 
 \[c_K^{d+e}=T_{e,d}(c_K^e\otimes 1) T_{d,e} (c_K^d\otimes 1)=(c_K^d\otimes 1)T_{e,d}(c_K^e\otimes 1)T_{d,e} \] 
 in $\mathcal{H}^B_{Q,q}(d+e)$, 
 where $c_K^d\otimes 1\in \HB\otimes\mathcal{H}^A_q(e)$ and $c_K^e\otimes 1\in \mathcal{H}^B_{Q,q}(e)\otimes\HB$
 are viewed as elements in $ \mathcal{H}^B_{Q,q}(d+e)$
 via $\HB\otimes\mathcal{H}^A_q(e) \subseteq \mathcal{H}^B_{Q,q}(d+e)$
 and via $\mathcal{H}^B_{Q,q}(e)\otimes\mathcal{H}^A_q(d) \subseteq\mathcal{H}^B_{Q,q}(e+d)=\mathcal{H}^B_{Q,q}(e+d).$ But this is checked by a straightforward computation in the Hecke algebra 
$\mathcal{H}^B_{Q,q}(d+e)$.
\end{proof}

\begin{rem}\label{lemKx}
Let $K_{F(V_n)}$ be the $K$-matrix defined in \S~\ref{sec:RKe}. Then we have 
\[ K_{F(V_n)} = F(c_K). \]
\end{rem}

\begin{rem}\label{rembmc}
Strengthening the idea of a cylinder braided action pair is the notion of a \emph{braided module category} (see~\cite[\S 4.3]{Enriquez} and~\cite[\S~5.1]{Brochier}). 
A cylinder braided action pair ($\mathcal B,\mathcal A$) is equipped with a cylinder twist which can be thought of as a natural map $t_X:1*X\to 1*X$ (via $X=1*X$). A braided module comes equipped with a twist $b_{M,X}:M*X\to M*X$ natural on both $M\in\mathcal{B}, X\in\mathcal{A}$ with axioms that ensure the twist is compatible with the braiding on $\mathcal{A}$. Therefore, for a braided module $(\mathcal{B},b)$ over $\mathcal{A}$ and each $M\in\mathcal{B}$, the action pair $(M*\mathcal{A},\mathcal{A}$) is cylinder braided with $t_{M*X}=b_{M,X}$. 

Our category $\mathcal{P}_{Q,q}$ is a braided module category over $\mathcal{AP}_q$. In the setting of $U_{Q,q}^B$-modules with $Q,q$ generic, Kolb~\cite{KolbBMC} shows that the category of finite dimensional $U_{Q,q}^B$-modules is a braided module category over the category of finite dimensional $\Uq$-modules.
If we restrict to $\res (\mathcal{AP}_q)\subseteq\mathcal{P}_{Q,q}$, we can obtain the twist by letting $b_{Y,X}=c_{X,Y}(t_X \otimes \id_Y)c_{Y,X}$ for $Y\in \res (\mathcal{AP}_q), X\in \mathcal{AP}_q$. 
When $Q,q$ are generic, every object in $\mathcal{P}_{Q,q}$ is a direct summand of an object in $\res (\mathcal{AP}_q)$, so this is enough. 
In the non-generic case, we need to further show that $b_{Y,X}$ restricts to submodules. 
For this, we can work with duals of Schur algebras and essentially build a couniversal $K$-matrix (see \cite[Section 5]{HY} where they use the couniversal $R$-matrix to show that $\mathcal{AP}$ is braided monoidal). In order to streamline the contents of the paper, we skip the proof of this fact. 
\end{rem}

\section{Composition for two-parameter polynomial functors}\label{sec:composition}

Let $d$ be a non-negative integer and $e$ be a positive integer.

\subsection{The category $\apde$}

We now define a category of (type A) quantum polynomial functors $\apde$ where composition is possible. This category is studied in \cite{BuciumasKo1}.
 
Recall the $e$-Schur algebra and the $e$-Hecke algebra defined in Section \ref{sec:RKe}. Let $\mathcal{C}^A_{d,e}$ be the category defined as follows: its objects are finite dimensional $S_{q}^A(n;e)$-modules (or the degree $e$ representation of $U_q(\mathfrak{gl}_n)$)) for all positive $n$. The morphisms are given by 
\[ \operatorname{Mor}(V, W):= \Hom_{\mathcal{H}_q^A(d,e)}(V\td, W\td),  \]
where the $e$-Hecke algebra acts on $V\td$ as in \S \ref{sec:RKe}. 
Define $\apde:=\mod_{\mathcal{C}^A_{d,e}}$.

Then \cite[Theorem 5.2]{BuciumasKo1} shows that there is a composition $\circ_A$ on $\mathcal{AP}^{*,*}_q$. More precisely this means that given $F \in \mathcal{AP}_q^{d_2,d_1e}, G \in \mathcal{AP}_q^{d_1,e}$, then we have $F \circ_A G\in \mathcal{AP}_q^{d_1 d_2,e}$. One can also check that $\circ_A$ is associative.

\subsection{The category $\mathcal{P}_{Q,q}^{d,e}$} \label{subsec:pQqde}

Define the category $\mathcal{C}^B_{d,e}$ as follows: its objects are finite dimensional $S_{Q,q}^B(n;e)$-modules, for all positive $n$. The morphisms are given by 
\[ \operatorname{Mor}(V, W):= \Hom_{\mathcal{H}_{Q,q}^B(d,e)}(V\td, W\td),  \] 
where the action of $\mathcal{H}_{Q,q}^B(d,e)$ on $V\td$ is given in Section \ref{sec:RKe}.  
Define $\mathcal{P}_{Q,q}^{d,e}:=\mod_{\mathcal{C}^B_{d,e}}$.

It is proved in \cite{BuciumasKo1}, assuming $q$ generic, that the category $\mathcal{AP}_q^{d,e}$ is equivalent to the category  $\operatorname{mod} \End_{\mathcal{H}_{q}^{A}(d,e)}((\bigoplus_{i=1}^d V_n^{\otimes e})\td )$ when $n\geq de$.  
One can prove a similar theorem in the type B setting:
\begin{thm}\label{thm:epfrepresentability}
Let $k=\mathbb C$ and $Q,q\in \mathbb C^\times$ generic. If $n \geq 2de$, the category $\mathcal{P}_{Q,q}^{d,e}$ is equivalent to the category of finite dimensional modules of the generalized Schur algebra 
\[S_{Q,q}^B( \bigoplus_{i=1}^d V_n^{\otimes e}; d):= \End_{\mathcal{H}_{Q,q}^B(d,e)}( (\bigoplus_{i=1}^d V_n^{\otimes e})^{\otimes d} ).\]
\end{thm}

We do not prove the theorem because the proof is long and tedious, and the techniques are the same as in the type A setting. See \cite[Corollary 6.14]{BuciumasKo1} for the type A argument which is similar. Note that the theorem requires semisimplicity, i.e. $Q,q$ have to be generic and $k$ has to be a field of characteristic $0$.

Let $F\in \mathcal{P}_{Q,q}^{d_2,d_1e}$  and $G \in \mathcal{AP}_q^{d_1,e}$. It is shown in \cite[Theorem 5.1]{BuciumasKo1} that $G(V)$ has the structure of an $S_q^A(n;d_1e)$-module. 

Recall that $F,G$ produce maps on morphism sets 
\[G: \Hom_{\mathcal{H}_{q}^A(d_1,e)}(V^{\otimes d_1}, W^{\otimes d_1}) \to \Hom(G(V), G(W))\]
for $V,W$ direct sums of $e$-Schur algebra-subquotients of $V_n^{\otimes e}$ for some $n$ (or $e$-Hecke pairs as they are called in~\cite{BuciumasKo1}), and 
\[F: \Hom_{\mathcal{H}_{Q,q}^B(d_2,d_1e)}(\bar V^{\otimes d_2}, \bar W^{\otimes d_2}) \to \Hom(F(\bar V), F(\bar W))\]
for $\bar V,\bar W$ direct sums of $ed_1$-Schur algebra-subquotients of $V_n^{\otimes ed_1}$. It seems (type B) $d_1e$-Hecke triples would be an appropriate name for such $\bar V,\bar W$. The reason for the use of ``triple'' is as follows: we are using the vector space structure of $\bar V, \bar W$, as well as their $R$-matrices and $K$-matrices to define the action of $\mathcal{H}_{Q,q}^B(d_2,d_1e)$ (for an $e$-Hecke pair we only needed the vector space structure and its $R$-matrix). 

Define $F \circ G \in \mathcal{P}_{Q,q}^{d_2d_1, e}$ as follows: for $V$ an $S_q^A(n;e)$-module set $F\circ G(V) :=F( G(V))$. This is well-defined since $G(V)$ has the structure of an $S_q^A(n;d_1e)$-module.
Define $F\circ G(x) \in \Hom (F\circ G(V), F\circ G (W))$ as the composition:
\begin{equation}\label{eq:comppsiF} 
\Hom_{\mathcal{H}_{Q,q}^B(d_1d_2,e)}(V^{\otimes d_1 d_2}, W^{\otimes d_1 d_2}) \xrightarrow{\Psi} \Hom_{\mathcal{H}_{Q,q}^B(d_2,d_1e)} (G(V)^{\otimes d_2}, G(W)^{\otimes d_2}) \xrightarrow{F} \Hom (FG(V), FG(W)), 
\end{equation}
where $\Psi$ is defined as follows: write  $x \in \Hom_{\mathcal{H}_{Q,q}^B(d_1d_2,e)}(V^{\otimes d_1 d_2}, W^{\otimes d_1 d_2})$ as 
\[ x=x_1\otimes x_2 \otimes \cdots \otimes x_{d_2}, \]
with $x_i \in \Hom_{\mathcal{H}_{q}^A(d_1,e)}(V^{\otimes d_1}, W^{\otimes d_1})$ and set $\Psi(x_1 \otimes \cdots \otimes x_{d_2}) := G(x_1)\otimes \cdots G(x_{d_2})$.   

\begin{lem}
The map $\Psi$ is well-defined. 
\end{lem}
\begin{proof}
Since $x \in \Hom_{\mathcal{H}_{Q,q}^B(d_1d_2,e)}(V^{\otimes d_1 d_2}, W^{\otimes d_1 d_2})$, it follows that $x$ commutes with the generators of $\mathcal{H}_{Q,q}^B(d_2,d_1e) \subset \mathcal{H}_{Q,q}^B(d_2d_1, e)$ and therefore $G(x_1)\otimes \cdots \otimes G(x_{d_2}) \in \Hom_{\mathcal{H}_{Q,q}^B(d_2,d_1e)} (G(V)^{\otimes d_2}, G(W)^{\otimes d_2})$.  
\end{proof}

The following theorem is a consequence of the fact that both maps in equation~\eqref{eq:comppsiF} are $k$-linear: 

\begin{thm}
The composition $F \circ G$ is a well-defined polynomial functor in $\mathcal{P}_{Q,q}^{d_2d_1, e}$. 
\end{thm}

The composition defined above is restated as follows in the language of Section~\ref{sec:cylindertwist}. 
Define $\mathcal{AEP}_q := \bigoplus_{d,e}\mathcal{AP}_q^{d,e}$. The composition $\circ_A$ is extended to $\mathcal{AEP}_q \times \mathcal{AEP}_q \to \mathcal{AEP}_q$ by setting 
\[ \circ_A  (\mathcal{AP}_q^{a,b} \times \mathcal{AP}_q^{d,e} ) = 0 \mbox{ if } b\neq de.   \]
There is an element $\id_{\mathcal{AP}_q} \in \mathcal{AEP}_q$ given by 
\[ \id_{\mathcal{AEP}_q} := \sum_{e} \id_{\mathcal{AP}_q^{1,e}},\]
where $id_{\mathcal{AP}_q^{1,e}}$ is the identity functor mapping an $e$-Hecke pair to itself. The category $\mathcal{AEP}_q$ with the operation $\circ_A$ and the element $\id_{\mathcal{AEP}_q}$ form a monoidal category. 

In the same way we extend the map $\circ: \mathcal{P}_{Q,q}^{d_2,d_1e} \times \mathcal{AP}_q^{d_1,e} \to \mathcal{P}_{Q,q}^{d_2d_1, e}$ to 
\[ \circ: \mathcal{EP}_{Q,q} \times \mathcal{AEP}_q \to \mathcal{EP}_{Q,q},\]
where $\mathcal{EP}_{Q,q}:=\bigoplus_{d,e} \mathcal{P}_{Q,q}^{d,e}$. The following proposition becomes a routine check:

\begin{prop}\label{prop:compositionaction}
The pair $(\mathcal{EP}_{Q,q}, \mathcal{AEP}_q)$  with action given by composition $\circ$ is an action pair. 
\end{prop}

\begin{rem}\label{rem:highercylinder}
It is shown in \cite{BuciumasKo1} that $\mathcal{EAP}_q$ has a $k$-(bi)linear tensor product $\otimes$ which is braided. 
Thus, one can extend the result of Section \ref{sec:cylindertwist} to the setting of this section. 
That is, the tensor product $\otimes$ on $\mathcal{EAP}_q$ extends to a $k$-linear action of $\mathcal{EAP}_q$ on $\mathcal{EP}_{Q,q}$; the objects in $\mathcal{EAP}_q$ restricts to the category $\mathcal{EP}_{Q,q}$; the action pair $(\res(\mathcal{EAP}_q),\mathcal{EAP}_q)$ thus obtained
is cylinder braided. 
The cylinder twist in this setting arises from the action of the elements
\[c_K^d(e)=\prod_{i=1}^d K_i(e) \in \mathcal{H}^B_{Q,q}(d,e).\] 
Above we used the notation $K_{i+1}(e)=T_{w_i}\cdots T_{w_1}T_{w_0}T_{w_1}\cdots T_{w_i}$, where $w_i,w_0$ are as in equations~\eqref{def:wi}, \eqref{def:w0}.  
\end{rem}

\section{Quantum symmetric powers and quantum exterior powers}\label{sec:pmsymext}

The easiest example of a polynomial functor is $\otimes^d \in \res\mathcal{P}_q^d \subseteq\mathcal{P}_{Q,q}^d$ which maps $V_n$ to $V_n\td$. 
In this section, we define important basic objects in $\mathcal{P}_{Q,q}^d$, namely the quantum $\pm$-symmetric powers and quantum $\pm$-exterior powers which supply examples of two-parameter polynomial functors outside $\res\mathcal{P}_q^d$. 
Consider $V_n\td$ as a representation of $\mathcal{H}_{Q,q}^B(d)$ on which the action of $T_i$ is given by \eqref{HeckeactionBC}. 
Note that the action of each generator $T_i \in \mathcal{H}^B_{Q,q}(d)$ on $V_n\td$ is diagonalizable with eigenvalues $q\inv$ and $-q$ for $T_i, i>0$ and $Q\inv$ and $-Q$ for $T_0$. 

In $\mathcal{P}_q^d$, we have the exterior power and symmetric power defined as
\begin{equation}\label{def:symext}
\begin{split}
\wedge^d V_n &= V_n\td/ \{(T_i+q)w\ |\ w\in V_n\td, i>0 \}; \\
S^d V_n &= V_n\td/\{ (T_i-q\inv)w\ |\ w \in V_n\td, i >0 \}.
\end{split}
\end{equation}
We generalize equation~\eqref{def:symext} using the $\HB$ action.

\begin{defn}\label{def:pmsepowers}
The quantum $\pm$-exterior powers $\wedge^d_{\pm}$ and the quantum $\pm$-symmetric powers $S^d_{\pm}$ are defined on each $V_n$ as
\begin{equation}\label{def:pmsymext}
\begin{split}
\wedge_-^d V_n &= V_n\td/ \{(T_0+Q)w, (T_i+q)w\ |\ w\in V_n\td, i>0 \}; \\
S_+^d V_n &= V_n\td/\{ (T_0-Q\inv)w, (T_i-q\inv)w\ |\ w \in V_n\td, i >0 \}; \\
\wedge_+^d V_n &= V_n\td/ \{(T_0-Q\inv)w, (T_i+q)w\ |\ w \in V_n\td, \ i>0\};\\
S_-^d V_n &= V_n\td/ \{ (T_0+Q)w, (T_i-q\inv)w\ |\ w\in V_n\td, i>0\}.
\end{split}
\end{equation}
\end{defn}

Given a map $f \in \Hom_{\mathcal{H}_{Q,q}^B}(V_n^{\otimes d}, V_m^{\otimes d})$, it follows by definition that $f(T_i+q) = (T_i+q)f$ and $f (T_i-q\inv) = (T_i-q^{-1})f$. The function $f$ can then be restricted to a map $f_\pm^S:S^d_{\pm} V_n \to S^d_{\pm} V_m$, or to a map $f_\pm^\wedge:\wedge^d_{\pm} V_n \to \wedge^d_{\pm} V_m$ by Definition \ref{def:pmsepowers}. The assignment $f\mapsto f^S_\pm$ (or $f_\pm^\wedge$) is a linear map $\Hom_{\mathcal{H}_{Q,q}^B(d)}(V_n\td, V_m\td) \to \Hom(S^d_\pm V_n, S^d_\pm V_m)$ (or $\Hom(\wedge^d_\pm V_n, \wedge^d_\pm V_m)$) on the morphism spaces. Therefore we have the following result.

\begin{prop}
The quantum $\pm$-exterior powers $\wedge^d_\pm$ and the quantum $\pm$-symmetric powers $S^d_\pm$ are polynomial functors.
\end{prop}

\begin{rem}
We define the four functors as quotients of $\otimes^d$. But in fact, they all split, and we may also view them as subfunctors.
We additionally introduce the following polynomial functors, the $\pm$-divided powers, by dualizing the definition of the $\pm$-symmetric powers. They are isomorphic to $\pm$-symmetric powers in our setting, but not in general (see Section 8). 
\begin{equation}\label{def:pmsymextsub}
\begin{split}
\Gamma_+^d V_n &= \{w\in V_n\td\ |\ (T_0-Q\inv)w=0, (T_i-q\inv)w=0, i >0 \}; \\
\Gamma_-^d V_n &= \{w\in V_n\td\ |\ (T_0+Q)w=0, (T_i-q\inv)w=0, i>0\}.
\end{split}
\end{equation}
\end{rem}

We describe a basis of each quantum exterior and symmetric power (evaluated at $V_n$).

Given $\boldsymbol{a}=({a}_1,\cdots,{a}_d)$ with ${a}_i \in \mathbb{I}_{n}$, we denote by $v(\boldsymbol{a})$ the standard vector $v_{a_1}\otimes\cdots\otimes v_{a_d}$ in $V_n\td$. 
We introduce the classes of vectors (depending on a pair of signs $\alpha,\beta\in\{\pm\}$)
\[v(\boldsymbol{a})_{\alpha \beta}:=\sum_{w\in W_d^B/\operatorname{Stab}_{W_d^B}(\boldsymbol{a}) } (\alpha Q)^{-\alpha\ell_0(w)}(\beta q)^{-\beta\ell_1(w)}v(w\boldsymbol{a}), \]
where the length functions $\ell_0, \ell_1$ are as in \eqref{eq:l0l1}.

\begin{prop}\label{basissympower}
The following hold:
\begin{enumerate}
\item The image of the set
$\{v(\boldsymbol{a})_{++}\ |\  0\leq{a}_1\leq\cdots \leq a_d, {a}_i \in \mathbb{I}_{n}\}$
is a basis of $S^d_+ V_n$. 
\item The image of the set
$\{v(\boldsymbol{a})_{+-}\ |\  0<{a}_1\leq\cdots \leq a_d, {a}_i \in \mathbb{I}_{n}\}$
is a basis of $S^d_- V_n$.
\item The image of the set
$\{v(\boldsymbol{a})_{-+}\ |\  0\leq{a}_1<\cdots < a_d, {a}_i \in \mathbb{I}_{n}\}$
is a basis of $\wedge^d_+ V_n$.
\item The image of the set
$\{v(\boldsymbol{a})_{--}\ |\  0<{a}_1<\cdots < a_d, {a}_i \in \mathbb{I}_{n}\}$
is a basis of $\wedge^d_- V_n$
\end{enumerate}
\end{prop}

\begin{proof}
We give an argument for $S^d_+$; the rest is similar and left to the reader.

We first check that the (image of the) set $\{ v(\boldsymbol{a})\}$, with $\boldsymbol{a}$ such that $ 0\leq a_1\leq\cdots \leq a_d, a_i \in \mathbb{I}_{n}$, spans $S^d_+V_n$.
In fact, for any standard vector $v(\boldsymbol{b})$ with $\boldsymbol{b}\in \mathbb I^d$ we can write $\boldsymbol{b}=w\boldsymbol{a}$ with $\boldsymbol{a}$ as above. For any reduced expression $s t\cdots u$ of $w\in W^B_d$, we have $v_{\boldsymbol{b}}=T_sT_t\cdots T_u v(\boldsymbol{a})=T_w v(\boldsymbol{a})$ because each $T_{s_i}$ action falls into the second case in \eqref{eq:Rmatrixdef},\eqref{eq:Kmatrixdef}.
So in $S^d_+ V_n$, the image of $v(\boldsymbol{b})$ is a multiple of the image of $v_{\boldsymbol{a}}$.

Inside $V\td$, the set $\{ v(\boldsymbol{a})_{++}\ |\ 0\leq{a}_1\leq\cdots \leq{a}_d, {a}_i \in \mathbb{I}_{n}\}$ is linearly independent and consists of eigenvectors for $T_i$ (for all $i$ at the same time). All $T_i$'s with $i>0$ have eigenvalue $q\inv$ and $T_0$ has eigenvalue $Q\inv$. Since $S^d_+V_n$ has the same dimension as $\Gamma^d_+V_n$, which is the submodule of $V_n\td$ spanned by $q\inv$ eigenvectors for $T_i, i>0$ and $Q\inv$ eigenvectors for $T_0$, this implies that the order of the set is smaller than the dimension of $S^d_+V_n$.

Combining the two paragraphs, we confirm that the images of $v(\mathbf{a})$ in $S^d_+V_n$ form a basis.
\end{proof}

\begin{rem}\label{stddimB}
Proposition \ref{basissympower} implies, for each $n$, the dimension of $\wedge^d_\pm V_n$, $S^d_\pm V_n$ does not depend on $q$ and $Q$. The dimension in each case has an easy formula depending on the parity of $n$:
\begin{equation}
\begin{aligned}[c]
&\dim \wedge_{\pm}^d V_{2r}={r \choose d}, \\
&\dim \wedge_{+}^d V_{2r+1}={r+1 \choose d}, \\
&\dim S^d_+V_{2r+1}={r+d\choose d},
\end{aligned}
\qquad \qquad \quad
\begin{aligned}[c]
&\dim S_{\pm}^d V_{2r}= {r+d-1\choose d},\\
&\dim \wedge_{-}^d V_{2r+1}={r \choose d},\\
 &\dim S^d_-V_{2r+1}={r+d-1 \choose d}. \\
\end{aligned}
\end{equation}
\end{rem}

\subsection{Higher degree quantum $\pm$-symmetric and exrerior powers}\label{BZstuff}

We now define higher version of the $\pm$-symmetric and $\pm$-exterior powers that live in the category $\mathcal {EP}_{Q,q}$ defined in Section~\ref{sec:composition}. The construction follows the idea in Berenstein and Zwicknagl~\cite{BZ} and makes crucial use of Proposition~\ref{eigenvaluesKi}. 

The eigenvalues of $c_K \in \mathcal{H}_{Q,q}^B(e)\subseteq \mathcal{H}_{Q,q}^B(d,e)$ are of the form $Q^i q^j$ and $-Q^iq^j$ for $i,j\in\mathbb Z, -e \leq i \leq e, -(e-1)e \leq j \leq (e-1)e$; this follows immediately from Proposition~\ref{eigenvaluesKi}. 
In order to be able to define positive and negative eigenvalues of $c_K$, we need to assume 
\[Q^i q^{j} \neq -1 \mbox{ for any } i,j \in \mathbb{Z} \mbox{ such that } -2e \leq i \leq 2e, -2(e-1)e \leq j \leq 2(e-1)e.\] This assumption is covered under our $Q,q$ generic assumption which will be enforced for the rest of the section. 

Then the two sets $\{Q^i q^j\}$ and $\{-Q^iq^j\}$ are disjoint; we call elements of the former set positive eigenvalues of $c_K$ and elements of the latter set negative eigenvalues of $c_K$. 
It is known that the eigenvalues of $T_{w_i} \in \mathcal{H}_{Q,q}^B(d,e)$ are of the form $\pm q^i$, this follows for example from~\cite[Lemma 1.2]{BZ}. This allows us to also partition the eigenvalues of $T_{w_i}$ into positive eigenvalues (of the form $+q^i$) and negative eigenvalues (of the form $-q^i$), again with no overlap between the two sets when $Q,q$ are generic. 

\begin{defn}
Given $V\in \mathcal{C}^B_{d,e}$ an $e$-Hecke triple as defined in \S~\ref{subsec:pQqde}, then
\begin{enumerate}
    \item let $S^{d,e}_+ V$ be the largest quotient of $\otimes^d V$ where each $T_{w_i}$ and $c_K$ have positive eigenvalues;
\item let $S^{d,e}_- V$ be the largest quotient of $\otimes^d V$ where each $T_{w_i}$ has negative eigenvalues and $c_K$ has positive eigenvalues;
\item let $\wedge^{d,e}_+ V$ be the largest quotient of $\otimes^d V$ where each $T_{w_i}$ has positive eigenvalues and $c_K$ has negative eigenvalues;
\item let $\wedge^{d,e}_- V$ be the largest quotient of $\otimes^d V$ where each $T_{w_i}$ and $c_K$ have negative eigenvalues.
\end{enumerate}
\end{defn}

Since the definition is natural on $V$, our $S^{d,e}_\pm$ and $\wedge^{d,e}_\pm$ are quotient functors of $\otimes$ and therefore the following proposition holds:

\begin{prop}
The functors $S^{d,e}_\pm$ and $\wedge^{d,e}_\pm$ belong to $\mathcal{P}^{d,e}_{Q,q}$.
\end{prop}
Note that $T_{w_i}$ and $c_K$ are not diagonalizable in general; the higher degree $\pm$-powers are generalized eigenspaces, not eigenspaces.

\begin{rem}
We do not know the dimension of the higher degree quantum $\pm$ symmetric and exterior powers. Even in the type A setting developed by Berenstein and Zwicknagl, the dimensions are not known in general. It is known that the dimension is less than or equal to the classical (q=1) dimension and in fact, it is mostly the case that $S_q^dV$ or $\wedge_q^d V$ have (strictly) smaller dimension than $S_{q=1}^dV$ or $\wedge_{q=1}^d V$. Thus we expect that the dimensions of $S^{d,e}_\pm V$ and $\wedge^{d,e}_\pm V$ also depend on the values of $Q,q$. 
\end{rem}

\section{Schur polynomial functors} \label{sec:schur}
The category $\PBd$ is semisimple, and the classification of simple objects is given by the Schur-Weyl duality.
In this section, we construct the simple objects explicitly in $\otimes^d$.

We first recall the type A quantum Schur functors from \cite{HH, HY}.
Given a partition $\lambda=(\lambda_1,\cdots,\lambda_r)$, let
\[ \wedge^\lambda := \wedge^{\lambda_1} \otimes \wedge^{\lambda_2} \otimes \cdots \otimes \wedge^{\lambda_r}, \]
\[ S^\lambda := S^{\lambda_1} \otimes S^{\lambda_2} \otimes \cdots \otimes S^{\lambda_r} ,\]
where $S^d, \Lambda^d$ are defined in equation~\eqref{def:symext}. 

We also write
\[\otimes ^\lambda=\otimes^{\lambda_1}\otimes \cdots\otimes \otimes^{\lambda_r}\]
even if $\otimes^\lambda\cong \otimes^d$ for any $\lambda\vdash d$.
For a partition $\lambda$ of $d$, the Schur functor $S_\lambda$ is defined as the image of the composition 
\begin{equation}\label{schurfunctor}
    s^A_\lambda:\wedge^{\lambda '} \xrightarrow{\iota_{\lambda'}}\otimes^d \xrightarrow{T_{c(\lambda)}}\otimes^d\to S^\lambda,
\end{equation}
where $\lambda'$ denotes the transpose of $\lambda$. The first map is given, on the evaluation at $V_n$ by 
\begin{equation}\label{wedgeinclusionexplicitformula}
\iota_{\lambda'}: v_{a_1}\wedge\cdots\wedge v_{ a_d}\mapsto \sum_{w\in S_{\lambda_1}\times\cdots \times S_{\lambda_r}\subseteq S_d}(-q)^{l(w)}v_{w\bold a},
\end{equation}
for $\bold{a}=(a_1,\cdots,a_d)$ with $0<a_1<\cdots<a_d$. The second map is the conjugation $T_{c(\lambda)}:V_n^{\otimes\lambda'}\to V_n^{\otimes \lambda}$. (The conjugation $c(\lambda)$ reads the column of the standard tableau corresponding to $\lambda$; if $\lambda=(4,2)$ then $c(\lambda)$ is the permutation ${(1,2,3,4,5,6)}\mapsto {(1,5,2,6,3,4)}$.) Note that since $S_{\lambda_1}\times\cdots\times S_{\lambda_r}$ is a parabolic subgroup of $S_d$, there is no ambiguity on the Coxeter length $l(w)$. 

Then the following statements are true under our assumption
\begin{enumerate}
\item the Schur functors $S_\lambda$ are irreducible; 
\item any irreducible in $\mathcal{AP}_q^d$ (the category of degree $d$ polynomial functors in type A), is isomorphic to $S_\lambda$ for some $\lambda\vdash d$; 
\item if $n\geq d$, then any irreducible for the quantum Schur algebra $S_q^A(n;d)$ is isomorphic to some $S_\lambda V_n$.
\end{enumerate}

\begin{rem}\label{rem:rouschurfunctor}
When $q$ is a root of unity, the $S_\lambda$ are not irreducible. One should instead understand the $S_\lambda$ in the following context: 
the category $\mathcal{AP}_q$ (or the polynomial representations for $U_q(\mathfrak{gl}_\infty)$ in the sense analogous to \S \ref{stability}) is highest weight where $S_\lambda$ are the costandard objects.
The dual definition \[\Gamma^\lambda\to  \otimes^d\to \wedge^{\lambda'}\] gives the Weyl functors which are the standard objects.
\end{rem}

The quantum definition of $S_\lambda$ is not immediately generalized to the coideal case because we cannot define the tensor products $\wedge_+^{a} \otimes \wedge_+^{b}$,  $S_+^{a} \otimes S_-^{b}$, etc. in our category.
The next three definitions bypass this difficulty.

Recall from Proposition~\ref{eigenvaluesKi} that $K_i$ has eigenvalues of the form $Q\inv q^{2j}$, $-Qq^{2j}$.

\begin{defn}\label{pluseigen}\label{defppower}
Let $\otimes_+^d$ be the largest quotients of $\otimes^d$ on which each $K_i$ has eigenvalues of the form $Q\inv q^{2j}$. 
Let $\otimes^d_-$ be the largest subfunctor of $\otimes^d$ on which each $K_i$ has eigenvalues of the form $-Qq^{2j}$.
\end{defn}

There is a small problem. The ``positive'' eigenvalues and the ``negative'' eigenvalues are still not well-defined. For example, if $q$ is a primitive $8$th root of unity and $Q=1$ then $Q\inv q^4=-1=-Qq^8$.
To make this definition valid, we need to impose a condition on $q,Q$ which we specify now.

\begin{prop}\label{fdandpmpower}
If 
\begin{equation*}
f_d(Q,q) :=\prod_{i=1-d}^{d-1} (Q^{-2}+q^{2i}) \neq 0, 
\end{equation*} 
then Definition~\ref{pluseigen} is well-defined.
\end{prop}
\begin{proof}
 If $f_d(Q,q)\neq 0$ then $f_i(Q,q)\neq 0$ for all $i\leq d$. The claim follows from the following lemma whose proof is elementary algebra and omitted.
\end{proof}

\begin{lem}\label{lemma:positivengativeev}
The set $\{-Qq^{2j}||j|<i\}$ and the set $\{Q\inv q^{2j}||j|<i\}$ are disjoint if and only if $f_i(Q,q)\neq 0$. 
\end{lem}

This lead us to the following assumption which is needed to define the Schur functors and which we impose until the end of the section.

\begin{assumption}\label{universalassumption}
Let $k$ be a field. Let $Q,q\in k^\times$ be such that $f_d(Q,q)\neq 0$.
\end{assumption}

 
If $Q=q=1$, then Assumption \ref{universalassumption} is equivalent to $\operatorname{char}k\neq 2$, which is the classical setting to define the symmetric and exterior power. 
We think of Assumption~\ref{universalassumption} as a correct two-parameter quantization of the assumption $\operatorname{char}k\neq 2$.

The $\otimes^d_\pm$ provide the easiest examples of quantum polynomial functors that do not have an analogue in type A (take $d=1$ for example).

\begin{prop}
The functor $\otimes^d_\pm$ is a direct summand of $\otimes^d$. 
\end{prop}
\begin{proof}
The (evaluation at $V_n$ of the) functor $\otimes^d$ decomposes into generalized eigenspaces for $K_i$, in particular, into (generalized) ``positive'' eigenspaces and ``negative'' eigenspaces.
Since all $K_i$ commute (see Lemma~\ref{lem:kicommute}), their actions on $\otimes^d$ are simultaneously triangularizable. 
Such a triangularization realizes $\otimes^d_\pm$ as a direct summand of $\otimes^d$.
\end{proof}

Since $\otimes^d_\pm$ is a direct summand of $\otimes^d$, we have the projections and inclusions
\begin{equation}\label{eq:pi}
\begin{split}
& p_\pm : \otimes^d\to \otimes^d_\pm, \\
& i_\pm  : \otimes^d_\pm\to \otimes^d.
\end{split}
\end{equation}
whose names will be repeatedly abused throughout the section: we denote by $p_\pm$ any projection that is induced by $p_\pm$ by a pushout diagram.
We can show:
\begin{lem}\label{purelation}
$p_{\pm} (V_n) = V_n^{\otimes d} u^\pm_d$.
\end{lem}

\begin{proof}
Recall that $V_n\td$ decomposes into simultaneous eigenspaces for $K_i$, $i=1,\cdots,d$.
Using Assumption~\ref{universalassumption} and Lemma~\ref{lemma:positivengativeev}, we say an eigenvalue (of some $K_i$) is positive if it is of the form $Q\inv q^{2j}$ and negative if it is of the form $-Qq^{2j}$. Then we can say $p_+V_n\td$ is the positive eigenspace of $V_n\td$.
The image of $u_d^+ = \prod_{j=1}^{d} (K_j + Q)$ acting on $V_n\td$ by definition annihilates all $-Q$-eigenvectors of $K_i$, for any $i$. Therefore we have $p_{\pm} (V_n) \subseteq V_n^{\otimes d} u^\pm_d$.  

For the opposite inclusion we argue by contradiction. Recall the $K_i$'s commute with each other. Suppose there is $v\in V_n\td u^+_d$, an eigenvector for all $K_i$, which has a negative eigenvalue for some $K_i$. Let $j$ be the smallest such $i$, and (by Proposition~\ref{eigenvaluesKi}) let $m$ be an integer such that $vK_j=-Qq^{2m}v$.
Let $a$ be the eigenvalue of $K_{j-1}$ for $v$. By assumption, $a$ is positive, in particular is not of the form $-Qq^{2m'}$. Thus the vector
$w_{j-1}=(q\inv-q)(-Qq^{2m})v+(a+Qq^{2m})T_{j-1}v$ (see Lemma~\ref{newlemma} and its proof)
is in the $-Qq^{2m}$-eigenspace for $K_{j-1}$. The vector $w_{j-1}$ is not necessarily in $V_n^{\otimes d} u^\pm_d$, we do not require it to be.
Note that $w_{j-1}$ is again a simultaneous eigenvector for all $K_i$. Now construct for $j-2 \geq i \geq1$ the vector
$w_i=(q\inv-q)(-Qq^{2m})w_{i+1}+(a_i+Qq^{2m})T_{i}w_{i+1}$, where  $a_iw_{i+1}=w_{i+1}K_i$, inductively.
Then each $w_i$ is an $-Qq^{2m}$-eigenvector for $K_i$.
Since the only eigenvalues of $K_1$ are $-Q$ and $Q\inv$, its eigenvalue at $w_1$ needs to be $-Q=-Qq^{2m}$, that is $m=0$.
But this means $vK_j=-Qq^{2m}v = -Qv$, which contradicts $v\in V_n\td u^+_d$.

A similar argument works for $p_-$.
\end{proof}

Now we relate the $\otimes^d_\pm$ with the $\pm$-symmetric/exterior powers.

\begin{prop}\label{pmands6}
We have the pushout diagrams
\begin{equation}
\begin{aligned}[c]
\begin{tikzcd}[row sep=3em, column sep=4em]
\otimes^d \arrow[r]\arrow[d,twoheadrightarrow]\arrow[r,phantom, very near start]& \otimes^d_{\pm} \arrow[d,twoheadrightarrow,""] \\
S^d \arrow[r,"p_\pm"]& S^d_\pm
\end{tikzcd} 
\qquad \qquad \qquad
\begin{aligned}
\end{aligned}
\begin{tikzcd}[row sep=3em, column sep=4em]
\otimes^d \arrow[r]\arrow[d,twoheadrightarrow]\arrow[r,phantom, very near start]& \otimes^d_{\pm} \arrow[d,twoheadrightarrow,""] \\
\wedge^d\arrow[r,"p_\pm"]& \wedge^d_\pm
\end{tikzcd}
\end{aligned}
\end{equation}
 \end{prop}
\begin{proof}
We prove this for $S_+^d$. Since each $T_i$ with $i>0$ acts on $S_+^dV_n$ as $q\inv$, if $K_0$ acts as $Q\inv$ then $K_i$ acts as $q^{-2i+1}Q\inv$. So each $(K_i-Q)$ is invertible on $S_+^dV_n$.
\end{proof}

Proposition \ref{pmands6} suggests the following definition.
\begin{defn}
We define $S^\lambda_\pm$, $\wedge^\lambda_\pm$ by the pushout diagrams
\begin{equation}
\begin{aligned}[c]
\begin{tikzcd}[row sep=3em, column sep=4em]
\otimes^d \arrow[r]\arrow[d,twoheadrightarrow]\arrow[r,phantom, very near start]& \otimes^d_{\pm} \arrow[d,twoheadrightarrow,""] \\
S^\lambda\arrow[r,"p_\pm"]& S^\lambda_\pm
\end{tikzcd} 
\qquad \qquad \qquad
\begin{aligned}
\end{aligned}
\begin{tikzcd}[row sep=3em, column sep=4em]
\otimes^d \arrow[r]\arrow[d,twoheadrightarrow]\arrow[r,phantom, very near start]& \otimes^d_{\pm} \arrow[d,twoheadrightarrow,""] \\
\wedge^\lambda\arrow[r,"p_\pm"]& \wedge^\lambda_\pm
\end{tikzcd}
\end{aligned}
\end{equation}
\end{defn}

Let us construct an analogue of the tensor product $\otimes_+^a$ with $\otimes_-^b$ that is a polynomial functor in $\mathcal{P}_{Q,q}^d$.
Since  $\mathcal{P}_{Q,q}$ is a right module category over $\mathcal{AP}_{q}$, we can form $\otimes^b_-\otimes \otimes^a$ and $ \otimes^a_+\otimes \otimes^b$ in $\mathcal{P}_{Q,q}$.

\begin{defn}\label{defpmpower}
The signed tensor power $\pmpower$ is the image of the map
\[\otimes^b_-\otimes \otimes^a\xrightarrow{T_{b,a}} \otimes^a_+\otimes \otimes^b.\]
\end{defn}
By the previous definition and Lemma~\ref{purelation} we have 
\[\pmpower (V_n)=V_n\td   u_b^{-}T_{b,a}u_a^+. \]
With the help of Definition \ref{defpmpower}, we define $S^{(\lambda,\mu)}$ and $\wedge^{(\lambda,\mu)}$:

\begin{defn}
Let $S^{(\lambda,\mu)}$ be the image of the map
\[S^\mu_-\otimes S^\lambda\xrightarrow{T_{b,a}\circ (i_- \otimes \id)} S^\lambda\otimes S^\mu\xrightarrow{p_+\otimes \id} S^\lambda_+\otimes S^\mu;\]
and let $\wedge^{(\lambda,\mu)}$ be the image of the map
\[\wedge^\mu_-\otimes \wedge^\lambda\xrightarrow{T_{b,a}\circ (i_- \otimes \id)} \wedge^\lambda\otimes \wedge^\mu\xrightarrow{p_+\otimes \id} \wedge^\lambda_+\otimes \wedge^\mu.\]
Note that the tensor products of the objects and maps are well-defined because $\mathcal{P}_{Q,q}$ is a module category over the monoidal category $\mathcal{AP}_{q}$ as shown in Section~\ref{sec:action}. 

In other words, we have the following commutative diagrams where the left faces are the definition of $\pmpower$, and the right faces are the definitions of $S^{(\lambda, \mu)}$ and $\wedge^{(\lambda, \mu)}$, respectively.

\begin{equation}\label{sym}
\begin{tikzcd}[row sep=3em, column sep=4em]
& \otimes^b_- \otimes \otimes^a \arrow[rr]  \arrow[dd,hook,pos=0.7,"T_{b,a}"]&  & S^\mu_- \otimes S^\lambda  \arrow[dl,pos=0.7,twoheadrightarrow] \arrow[dd,hook,pos=0.7,"T_{b,a}"] \\
^a_+\otimes^b_- \arrow[rr,crossing over] \arrow[dd,hook] \arrow[from=ur,twoheadrightarrow]& & S^{(\lambda,\mu)} & \\
& \otimes^{a+b}  \arrow[rr] \arrow[dl,twoheadrightarrow]& & S^\lambda \otimes S^\mu  \arrow[dl,twoheadrightarrow] \\
\otimes_+^a\otimes \otimes^b  \arrow[rr] & & S^{\lambda}_+\otimes S^\mu \arrow[from=uu,hook,crossing over] &  \\
\end{tikzcd}
\end{equation}

\begin{equation}\label{wedge}
\begin{tikzcd}[row sep=3em, column sep=4em]
& \otimes^b_- \otimes \otimes^a \arrow[rr]  \arrow[dd,hook,pos=0.7,"T_{b,a}"]&  & \wedge^\mu_- \otimes \wedge^\lambda  \arrow[dl,pos=0.7,twoheadrightarrow] \arrow[dd,hook,pos=0.7,"T_{b,a}"] \\
^a_+\otimes^b_- \arrow[rr,crossing over] \arrow[dd,hook] \arrow[from=ur,twoheadrightarrow]& & \wedge^{(\lambda,\mu)} & \\
& \otimes^{a+b}  \arrow[rr] \arrow[dl,twoheadrightarrow]& & \wedge^\lambda \otimes \wedge^\mu  \arrow[dl,twoheadrightarrow] \\
\otimes_+^a\otimes \otimes^b  \arrow[rr] & & \wedge^{\lambda}_+\otimes \wedge^\mu \arrow[from=uu,hook,crossing over] &  \\
\end{tikzcd}
\end{equation}
\end{defn}

We have $\wedge^{(\lambda,\mu)}\in \mathcal P_{Q,q}$ and $S^{(\lambda,\mu)}\in \mathcal P_{Q,q}$.
Note that if $Q=q=1$, we have 
\[\wedge^{(\lambda,\mu)} \cong \wedge_+^{\lambda_1}\otimes\cdots\otimes\wedge_+^{\lambda_r}\otimes\wedge_-^{\mu_1}\otimes\cdots\otimes\wedge_-^{\mu_r}\] and 
\[S^{(\lambda,\mu)} \cong S_+^{\lambda_1}\otimes\cdots\otimes S_+^{\lambda_r}\otimes S_-^{\mu_1}\otimes\cdots\otimes S_-^{\mu_r}.\] 
Thus we may think of $\wedge^{(\lambda,\mu)}$ and $S^{(\lambda,\mu)}$ as deformed tensor products which are not tensor products in the usual sense, but devolve to the usual tensor product when $Q,q=1$. 

\begin{ex}
(d=2) We have
\[\otimes^2=S^{((1,1),0)}\oplus S^{((1),(1))}\oplus X \oplus S^{(0,(1,1))}\]
where $X$ is isomorphic to $S^{((1),(1))}$ and can for example be taken to be $V\otimes V/(T_0+Q,T_1T_0T_1-Q\inv)$ (here we want a strict decomposition, not up to isomorphism).  Note that for the bipartitions appearing here, there is no difference between $S$ and $\wedge$ (so we could have replaced $S^{((1),(1))}$ by $\wedge^{((1),(1))}$ in the equation above).
Furthermore, there is a decomposition
\[S^{(0,(1,1))}=\wedge^{(0,(1,1))}=\wedge^2_-\oplus S^2_-\]
and \[S^{((1,1),0)}=\wedge^{((1,1),0)}=\wedge^2_+\oplus S^2_+\]
into direct sum of irreducibles.
\end{ex}

{\scriptsize
\begin{figure}
\begin{tikzcd}[cramped,row sep=7.5em, column sep=1.2em]
& \otimes^b_-\otimes \otimes^a \arrow[dl,hook,] \arrow[dr,twoheadrightarrow,] \arrow[rrr,]  \arrow[rrrrrrrrr,bend left=20] 
& & &S^\mu_- \otimes S^\lambda   \arrow[dr,twoheadrightarrow]
\arrow[dl,hook,crossing over]
& & &\arrow[lll,hook,] S_{(\varnothing,\mu)} \otimes S_{\lambda} \arrow[dl,hook] \arrow[dr,twoheadrightarrow]
& & & \arrow[lll, twoheadrightarrow] \wedge^{\mu'}_- \otimes \wedge^{\lambda'} \arrow[dl,hook] \arrow[dr,twoheadrightarrow] &  \\
\otimes^{a+b}  \arrow[rrrrrrrrr,bend left=20,crossing over, end anchor=north] \arrow[dr,twoheadrightarrow,pos=0.6] & &_+^a\otimes^b_-  \arrow[rrrrrrrrr,bend right=30,crossing over] \arrow[dl, hook] \arrow[rrr,bend left=30,crossing over]  &\arrow[from=lll,bend right=25, end anchor=south,crossing over,pos=0.1,below,]  
S^\lambda \otimes S^\mu  \arrow[dr,crossing over,twoheadrightarrow,pos=.6] \arrow[from=rrr,bend right=30,crossing over,hook,] &  &
S^{(\lambda,\mu)}  \arrow[dl,crossing over,hook] 
& S_\lambda \otimes S_\mu  \arrow[dr,twoheadrightarrow,crossing over,pos=0.6] &  & S_{(\lambda,\mu)} \arrow[dl,hook,crossing over]  \arrow[lll,bend left=30,crossing over,hook,pos=0.7] & \wedge^{\lambda'} \otimes \wedge^{\mu'} \arrow[dr,twoheadrightarrow,crossing over] \arrow[lll,bend right=20,crossing over,twoheadrightarrow,"s_\lambda^A\otimes s_\mu^A"] & &  \wedge^{(\lambda', \mu')} \arrow[dl,hook, start anchor=south]  \arrow[lll,bend left=20,crossing over, start anchor=south, twoheadrightarrow,pos=0.65, start anchor=south west] \\
& \otimes_+^a\otimes \otimes^b \arrow[rrrrrrrrr,bend right=20] \arrow[rrr]& & & S^\lambda_+\otimes S^\mu& & &\arrow[lll] S_{(\lambda,\varnothing)}\otimes S_\mu & & &\arrow[lll] \wedge^{\lambda'}_+ \otimes \wedge^{\mu'} &  \\
\end{tikzcd}
\caption{The diagram above consists of four diamonds and maps between them is used to define the Schur functor $S_{(\lambda,\mu)}$. }\label{smallvalentinstear}
\end{figure}
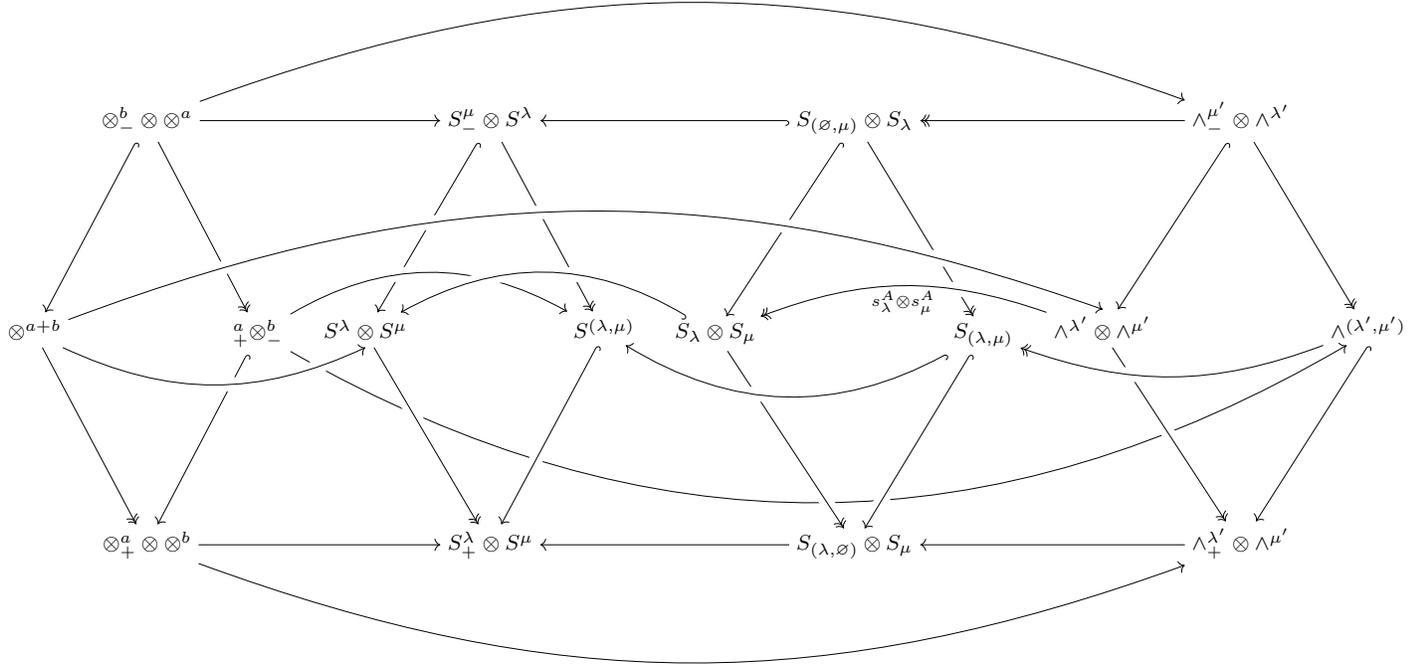
}

\begin{defn}\label{def:schurfunctor}
The Schur functor $S_{(\lambda,\mu)}$ is defined in the commutative diagram in Figure~\ref{smallvalentinstear}.
The two leftmost diagrams form a subdiagram equivalent to the diagram in \eqref{sym}, while the leftmost and rightmost diamonds form a subdiagram equivalent to the diagram in \eqref{wedge}. 
The rightward maps are induced from the definitions of symmetric and exterior power; the diamonds are induced from the definition of $\pmpower$. See also the diagrams \eqref{sym}, \eqref{wedge} which are subdiagrams of the diagram in Figure~\ref{smallvalentinstear}.
Then the leftward maps are induced from the map $s_\lambda^A\otimes s_\mu^A$ where $s^A_\lambda$ from \eqref{schurfunctor} defines the type A Schur functors. 

In particular, the Schur functor $S_{(\lambda,\mu)}$ can be defined as the image of the map:
\begin{equation}\label{eq:schurexactsequence}
\wedge^{(\lambda',\mu')} \xrightarrow{c^B_{(\lambda,\mu)}} \: ^a_+\otimes^b_- \xrightarrow{} S^{(\lambda, \mu)},
\end{equation}
where the right map is the projection in the diagram in Figure~\ref{smallvalentinstear} and the left map is induced from the map $c^A_\lambda \otimes c^A_\mu$ defined in equation~\eqref{wedgeinclusion}, where $c^A_\lambda = T_{c(\lambda)}\iota_{\lambda'}$ (see~\eqref{wedgeinclusionexplicitformula} and after).

\begin{equation}\label{wedgeinclusion}
\begin{tikzcd}[row sep=2em, column sep=3em]
& \wedge^{\lambda'}\otimes \wedge^{\mu'} \arrow[rr,hook,"c^A_\lambda \otimes c^A_\mu"]  \arrow[dd,pos=0.7]&  & \otimes^\lambda \otimes \otimes^\mu = \otimes^d  \arrow[dl,pos=0.7] \arrow[dd,pos=0.7] \\
\wedge_+^{\lambda'} \otimes \wedge^{\mu'} \arrow[rr,crossing over,hook] \arrow[dd] \arrow[from=ur]& & \otimes_+^a\otimes \otimes^b & \\
& \wedge^{\mu'}_- \otimes \wedge^{\lambda'}  \arrow[rr,hook] \arrow[dl]& & \otimes_-^b\otimes \otimes^a  \arrow[dl] \\
\wedge^{(\lambda',\mu')}  \arrow[rr,hook,"c^B_{(\lambda,\mu)}"] & & ^a_+\otimes^b_- \arrow[from=uu,crossing over] &  \\
\end{tikzcd}
\end{equation}

\end{defn}

\begin{ex}\label{extremeschurfunctor}
For $\varpi_d=(1,\cdots,1,0,\cdots,0)$, 
we have $S_{(\varpi_d,0)}=\wedge^d_+$ and $S_{(0,\varpi_d)}=\wedge^d_-$. 
For $d\varpi_1=(d,0,\cdots,0)$, we have $S_{(d\varpi_1,0)}=S^d_+$ and $S_{(0,d\varpi_1)}=S^d_-$.
\end{ex}

\subsection{Schur functors in generic case}

In this subsection, we relate the Schur functors with the Young symmetrizers in \S~\ref{subsec:youngsym}. For this, it is necessary to assume that $k=\mathbb C$ and $Q,q$ are generic.

\begin{prop}\label{lem:schuryoungsymsame}
We have for each $n$, $\lambda,\mu$
\[S_{(\lambda,\mu)}(V_n)\cong (V_n\td) e'_{\lambda,\mu}\] as $\SB(n;d)$-modules where $e'_{\lambda,\mu}$ is the Young symmetrizer defined in \eqref{defe'}. 
\end{prop}
\begin{proof}
The projection
$\otimes^d \to S_\lambda \otimes S_\mu$ is isomorphic to (acting with) the Young symmetrizer $e_\lambda \otimes e_\mu =(e_\lambda\otimes \id) (\id \otimes e_\mu)$. The projection $\otimes^d\xrightarrow{u_b^-T_{b,a}u^+_a} (\pmpower)$ from Definition \ref{defpmpower} is isomorphic to multiplication by $e_{a,b}$ from equation~\eqref{eq:eab}. By Lemma~\ref{purelation} we have that $\oplus^d_\pm=\oplus^du^d_\pm$.


The claim now follows from the Definition in Figure~\ref{smallvalentinstear} (note specifically the implicit square containing $\otimes ^d, ^a_+\otimes^b_-, S_\lambda \otimes S_\mu, S_{(\lambda,\mu)}$) and the fact that $e_{a,b}$ and $e_\lambda \otimes e_\mu$ are idempotents and commute.
\end{proof}

\begin{ex}\label{d=2schur}
($d=2$) There are five bipartitions $(\lambda,\mu)\vdash 2$, namely $((1,1),0)$, $(0,(1,1))$, $((2),0)$, $(0,(2))$, $((1),(1))$. 
The only case that is not covered in Example~\ref{extremeschurfunctor} is $((1),(1))$.
A defining sequence in this case is \[\wedge^{((1),(1))}\to \otimes^2\to S^{((1),(1))}.\]
One sees from the definition that $\wedge^{((1),(1))}=\ ^1_+\otimes^1_-=S^{((1),(1))}$ and that the composition is an isomorphism, hence we have $S_{((1),(1))}=S^{((1),(1))}\cong \wedge^{((1),(1))}$.
Thanks to the Schur-Weyl duality, we know that $\otimes^2$ has four distinct irreducible summands with multiplicity one and a unique (up to isomorphism) irreducible summand with multiplicity 2. The former correspond to $((1,1),0)$, $(0,(1,1))$, $((2),0)$, $(0,(2))$ and the latter is necessarily isomorphic to $S_{((1),(1))}$.
\end{ex}

Example \ref{d=2schur} generalizes to give the following description/classification of the irreducible polynomial functors in $\mathcal{P}_{Q,q}$.

\begin{thm}\label{thm:schurfunctortheorem}
The Schur functors $S_{(\lambda,\mu)}$ are irreducible, mutually non-isomorphic, and form a complete list of irreducibles in $\mathcal{P}_{Q,q}$. 
\end{thm}
\begin{proof}
 The claim follows from Proposition~\ref{lem:schuryoungsymsame}, Proposition~\ref{e'givesirred} and Proposition~\ref{thm:representability}.
\end{proof}

\begin{rem} 
We have that $S_{(\lambda,\mu)}(V_n) = L_{\lambda,\mu}(n)$. By \cite[Theorem 3.1.1]{LNX} and \cite[Theorem 6.19]{HH}, the dimension of the $\SB(n;d)$-module $S_{(\lambda,\mu)}V_n$ does not depend on $q,Q$. 
Thus it has a basis indexed by the set of semistandard bitableaux of shape $\lambda,\mu$. 
\end{rem}

\begin{rem}
It would be interesting to relate our construction of the irreducibles to the results of Watanabe~\cite{Watanabe}, where the author constructs crystal basis for irreducible representations of $U^B_{Q,q} (\mathfrak{gl}_n)$ for $n$ odd. 
\end{rem}

\subsection{Schur functors in non-generic case}

Theorem \ref{thm:schurfunctortheorem} is not true when $Q,q$ are roots of unity or $\operatorname{char}k>0$. But that is only because the formulation of the result is not the right one. (See Remark \ref{rem:rouschurfunctor}.) In this subsection, we place the Schur functors in the right context.

The category $\mathcal{P}_{Q,q}$ is semisimple under the assumption of Theorem \ref{thm:schurfunctortheorem} and therefore can be viewed as a highest weight category where the irreducible, standard and costandard objects coincide. Then Theorem \ref{thm:schurfunctortheorem} is equivalent to saying that
the Schur functors $S_{(\lambda,\mu)}$ give a complete list of mutually non-isomorphic costandard objects in $\mathcal P_{Q,q}$.

It is proved in \cite[Theorem 3.1.1]{LNX}, assuming $f_d(Q,q)\neq 0$, that $\SB(n;d)$ is quasi-hereditary for all $n,d$. Then by Theorem \ref{thm:representability}, the categories $\PBd$ and $\mathcal P_{Q,q}$ are highest weight. In that case, we expect that the $S_{(\lambda,\mu)}$ are the costandard objects in $\mathcal P_{Q,q}$ and the Weyl functors, which are defined by dualizing our definition of Schur functors, are the standard objects in $\mathcal P_{Q,q}$. We also expect that a direct proof of quasi-heredity using the Schur functors and Weyl functors similar to the approaches in \cite{ABW,KrauseHW} exists. 
We note that without the assumption $f_d(Q,q)\neq 0$, the algebra $\SB(n;d)$ is not quasi-hereditary in general (see \cite[Example 6.1.2]{LNX} and the remark thereafter). 

\subsection{Higher degree Schur functors}\label{subsec:higherschur}

We now assume $Q,q$ to be generic. 
Generalizing the functors $S^{d,e}_\pm,\wedge^{d,e}_\pm\in\mathcal{P}_{Q,q}^{d,e}$
defined in \S~\ref{BZstuff}, we can define Schur functors in $\mathcal{P}_{Q,q}^{d,e}$. We give an outline of this construction.

First define  $S^{\lambda,e}_+$ to be the largest quotient of $S^\lambda$ (here we denote by $S^\lambda$ the restriction of $S^\lambda=S^{\lambda_1}\otimes\cdots \otimes S^{\lambda_r}\in \mathcal{AP}_q^{d,e}$ to $\mathcal{P}_{Q,q}^{d,e}$) where $c_K^e\in \mathcal{H}^B_{Q,q}(d,e)$ has eigenvalues of the form $+Q^iq^j$, $i,j\in\mathbb Z$, and define similarly  $S^{\lambda,e}_-, \wedge^{\lambda,e}_\pm$. Then consider the higher degree analogue of the maps $T_{c(\lambda)}$ (see \eqref{schurfunctor}) and $T_{b,a}$ (see  Definition~\ref{defpmpower}), which are obtained by writing $T_{c(\lambda)}, T_{b,a}$ as a product of the standard generators $T_i$ in $\HB$  and replacing the $T_i$ with the higher degree generator $T_{w_i}\in \mathcal{H}^B_{Q,q}(d,e)$ (see \eqref{def:wi} and \eqref{def:w0}). 
The rest of the construction is now identical to that of the Schur functors in $\mathcal{P}^d_{Q,q}$ using Remark~\ref{rem:highercylinder}.

The higher degree Schur functors supply many non-trivial examples of polynomial functors in $\mathcal{P}^{d,e}_{Q,q}$.
Unlike in the case $e=1$, however, the Schur functors are decomposable in general. Their decomposition (even when $Q,q$ are generic) is a difficult and interesting problem. 
While we have little understanding on the higher degree Schur functors at the moment, we hope that they lead us to a structure theory of the categories $\mathcal{P}^{d,e}_{Q,q}$.

\bibliography{bib.bib}
\bibliographystyle{alpha}  

\end{document}